\documentclass[11pt]{amsart}

\usepackage{amssymb}
\usepackage{overpic}
\usepackage{enumitem}   
\usepackage{graphicx} 
\usepackage{amsrefs}
\usepackage{xcolor}
\usepackage[a4paper,margin=2cm]{geometry}
\usepackage[hidelinks]{hyperref}

\graphicspath{{Figures/}}

\newtheorem{theorem}{Theorem}

\newtheorem{proposition}{Proposition} 
 
\newtheorem{remark}{Remark} 
 
\newtheorem{definition}{Definition} 
\newtheorem{example}{Example} 

\begin{document}
	
\title[On the stability of hybrid polycycles]{On the stability of hybrid polycycles}

\author[Paulo Santana and Leonardo Serantola]{Paulo Santana$^1$ and Leonardo Serantola$^1$}

\address{$^1$ IBILCE--UNESP, CEP 15054--000, S. J. Rio Preto, S\~ao Paulo, Brazil}
\email{paulo.santana@unesp.br; l.serantola@unesp.br}



\begin{abstract}
	In this paper we provide the stability of generic polycycles of hybrid planar vector fields, extending previous known results in the literature. The polycycles considered here may have hyperbolic saddles, tangential singularities and jump singularities. 
\end{abstract}

\maketitle

\section{Introduction}\label{Sec1}

Let $X$ be a planar smooth vector field (i.e. of class $C^\infty$). A \emph{graphic} of $X$ is a compact, non-empty invariant subset which is a continuous (not necessarily homeomorphic) image of $\mathbb{S}^1$, composed by some singularities $p_1,\dots,p_n,p_{n+1}=p_1$ and distinct regular orbits $L_1,\dots,L_n$ such that $L_i$ is an orbit from $p_{i+1}$ to $p_i$, see Figure~\ref{Fig16}.
\begin{figure}[ht]
	\begin{center}
		\begin{overpic}[width=5cm]{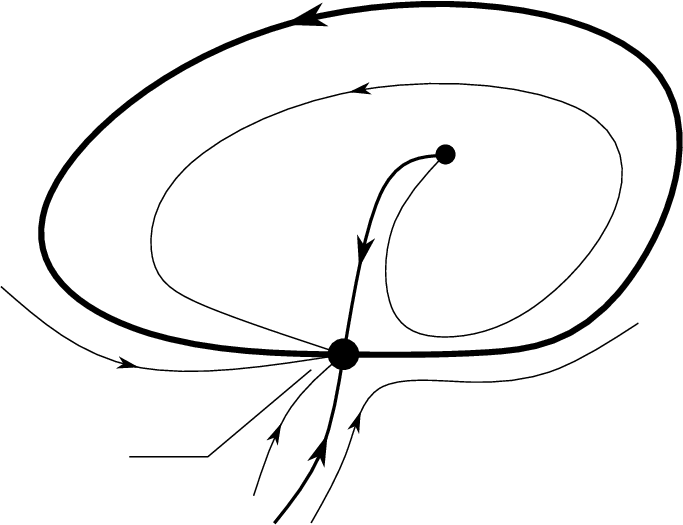}
		%\begin{overpic}[width=5cm,grid,tics=5]{Fig6.eps} 
			\put(10,64){$L_1$}
			\put(12,9){$p_1$}
		\end{overpic}
	\end{center}
\caption{Illustration of a graph given by the connection of the separatrix of a saddle-node with its parabolic sector.}\label{Fig16}
\end{figure}

A \emph{polycycle} is a graphic with a well defined first return map on one of its sides. See Figure~\ref{Fig7}$(a)$.
\begin{figure}[ht]
	\begin{center}
		\begin{minipage}{8cm}
			\begin{center} 
				\begin{overpic}[height=5cm]{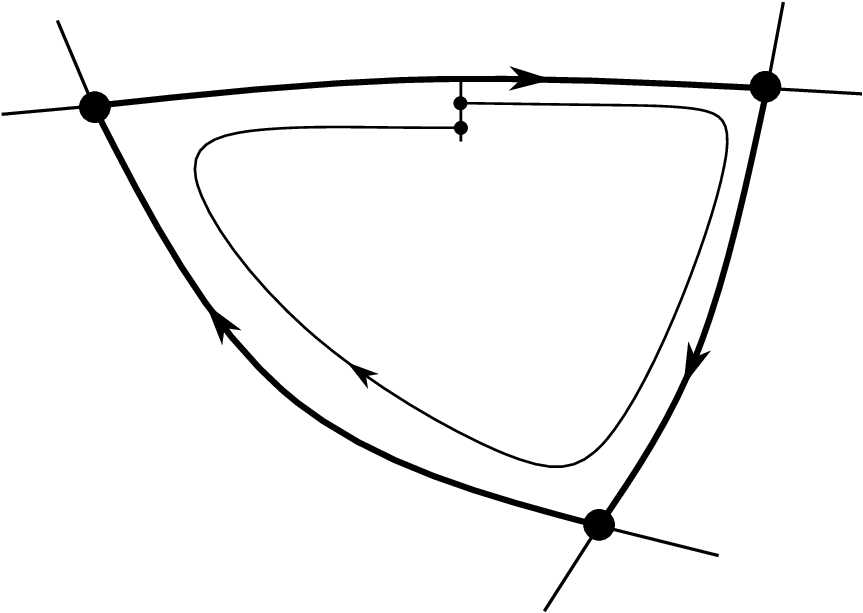} 
				%\begin{overpic}[height=5cm,grid,tics=5]{Fig13.eps} 
					\put(69,5){$p_1$}
					\put(91,63){$p_2$}
					\put(5,54){$p_3$}
					
					\put(85,30){$L_1$}
					\put(50,65){$L_2$}
					\put(27,22){$L_3$}
				\end{overpic}
				
				$(a)$
			\end{center}
		\end{minipage}
		\begin{minipage}{8cm}
			\begin{center} 
				\begin{overpic}[height=5cm]{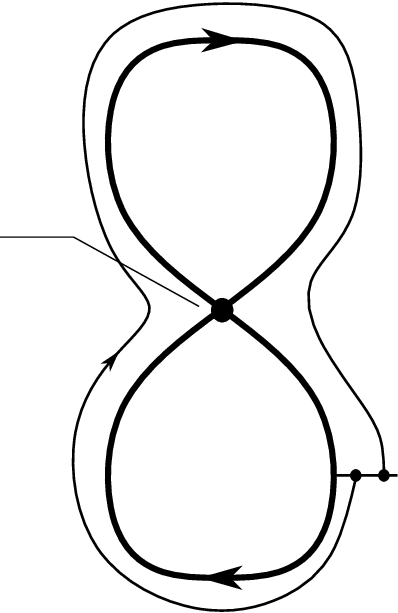} 
				%\begin{overpic}[height=5cm,grid,tics=5]{Fig7.eps} 
					\put(-10,64){$p_1=p_2$}
					\put(45,75){$L_1$}
					\put(20,20){$L_2$}
				\end{overpic}
				
				$(b)$
			\end{center}
		\end{minipage}
	\end{center}
\caption{Illustration of a polycycle $(a)$ homeomorphic and $(b)$ not homeomorphic to $\mathbb{S}^1$.}\label{Fig7}
\end{figure}
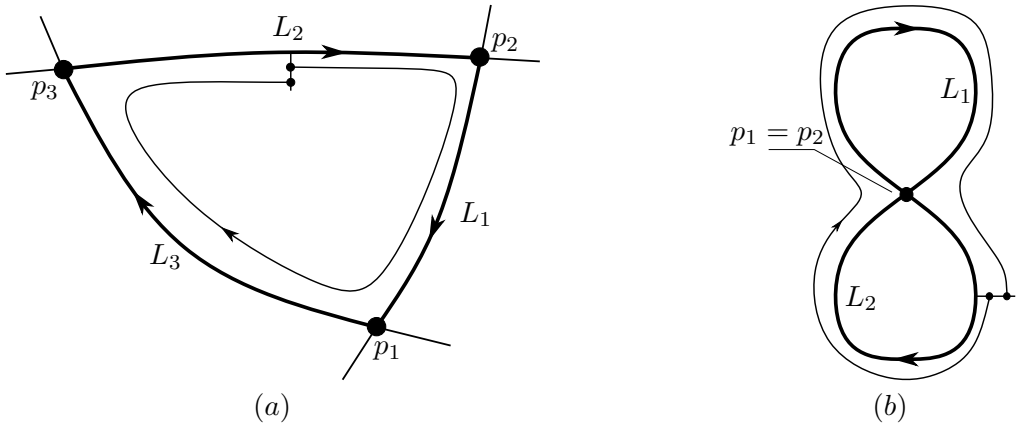
In addition with singularities and periodic orbits, it follows from the Poincar\'e-Bendixson Theorem (see Perko~\cite[Section~$3.7$]{PerkoBook}) that graphics are the only possible limit sets of a bounded orbit of $X$. The study about polycycles goes back to Poincar\'e and Dulac \cite{Dulac} in the first half of the $20$th century. In recent years there were many studies about graphics and polycycles, dealing for example with its cyclicity \cite{Mourtada,Mourtada2,Mourtada3,Mourtada4,SheHanTia2020,HanWuBi2004,HanZhu2007,MarVil2022,Duk2023,BGS}, stability \cite{GasManMan,HanHuLiu2003} and the bifurcation of critical periodic orbits \cite{MarVil2020,MarVil2021}. In particular Cherkas~\cite{Cherkas} characterized the stability of polycycles under generic conditions. For more details about the theory of polycycles in smooth vector fields, we refer to Roussarie~\cite[Chapter $5$]{Roussarie} and the references therein. 

With the development of the \emph{Filippov convention}~\cite{Filippov} for non-smooth planar vector fields in the second half of the $20$th century, it was natural to extend known results in the literature for smooth vector fields, to the Filippov one. On the wake of such effort, Buzzi et al~\cite{BuzCarEuz2018} extended the Poincar\'e-Bendixson Theorem to the Filippov vector fields and proved that (pseudo) graphics and polycycles remain a very important limit set on Filippov systems, see Figure~\ref{Fig17}. 
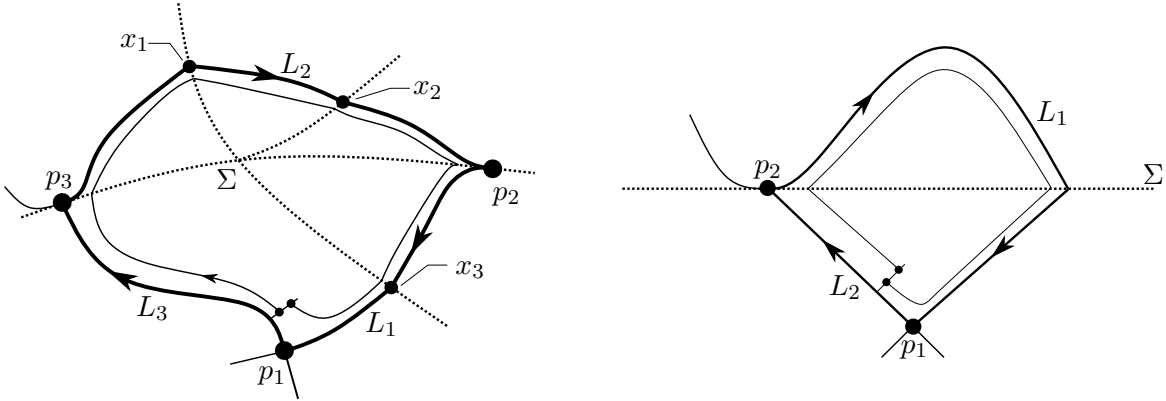
\begin{figure}[ht]
	\begin{center}
		\begin{minipage}{8cm}
			\begin{center} 
				\begin{overpic}[width=7cm]{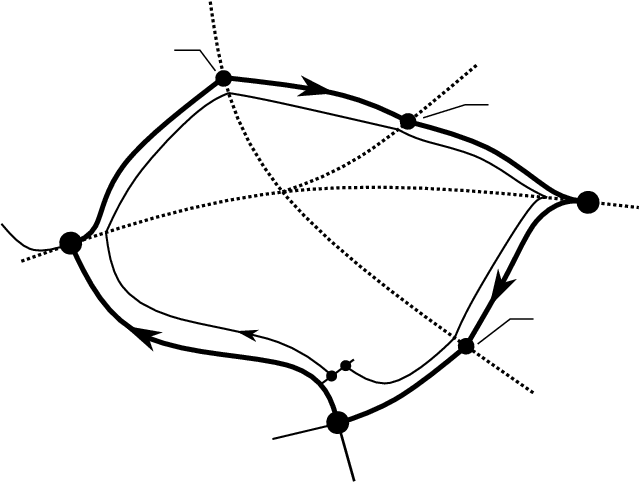} 
				%\begin{overpic}[width=7cm,grid,tics=5]{Fig14.eps} 
					\put(48,4){$p_1$}
					\put(92,38){$p_2$}
					\put(8,41){$p_3$}
					
					\put(68,13){$L_1$}
					\put(52,62){$L_2$}
					\put(25,16){$L_3$}
					
					\put(22,67){$x_1$}
					\put(77,58){$x_2$}
					\put(85,24){$x_3$}
					
					\put(40,40){$\Sigma$}
				\end{overpic}
			\end{center}
		\end{minipage}
		\begin{minipage}{8cm}
			\begin{center} 
				\begin{overpic}[width=7cm]{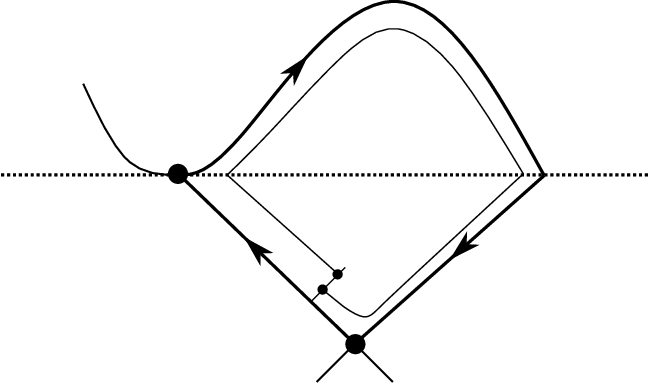} 
				%\begin{overpic}[width=7cm,grid,tics=5]{Fig8.eps} 
					\put(52.5,1){$p_1$}
					\put(25,35){$p_2$}
					\put(78,45){$L_1$}
					\put(39,12){$L_2$}
					\put(98,33){$\Sigma$}
				\end{overpic}
			\end{center}
		\end{minipage}
	\end{center}
	\caption{Illustration of a polycycle on planar Filippov systems.}\label{Fig17}
\end{figure}
It is also noted worthy to observe the recent work of Han and Zhou~\cite{Han3} characterizing the stability of limit cycles in Filippov systems, extending the classical results of Poincar\'e first return map in the smooth case. In recent years there were many studies about graphics and polycycles in Filippov systems, dealing for example with its bifurcation diagrams \cites{AndGomNov2023,AndJefMarTei2022,BuzCarTei2012,NovTeiZel,Han1,Han2} and regularization \cites{NovRon2,NovRon}. In particular the author in~\cite{San2023} extended the result of Cherka, characterizing the stability of polycycles in Filippov systems under generic conditions. For more details, we refer to Andrade, Gomide and Novaes~\cite{AndGomNov2023} and the references therein. 

With the recent development of the theory of \emph{hybrid systems}, it is natural to seek for a generalization of the concepts and tools of smooth and Filippov systems to the hybrid framework. Roughly speaking, a hybrid dynamical system is a system that can both \emph{flow} and \emph{jump}, see Figure~\ref{Fig1}.
\begin{figure}[ht]
	\begin{center}
		\begin{overpic}[width=8cm]{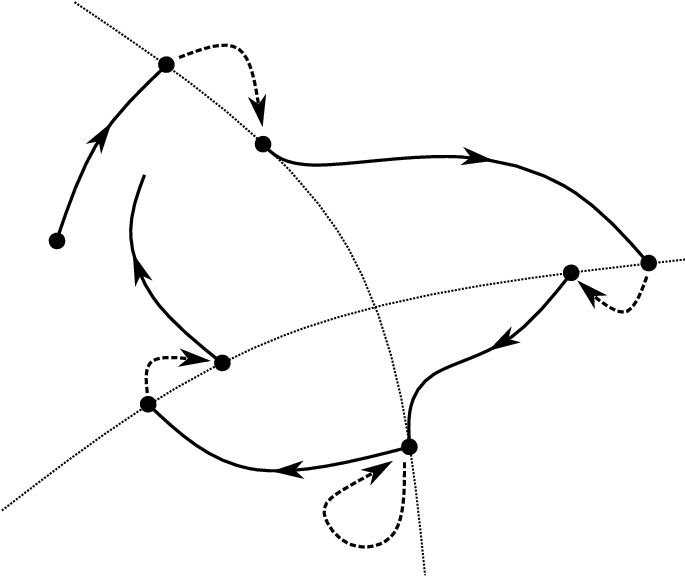}
		%\begin{overpic}[width=8cm,grid,tics=5]{Fig1.eps} %Esse é só para colocar os \put
			\put(7,45){$q$}
			\put(23,78){$p$}
			\put(35,59){$\overline{p}$}
			\put(38,75){$\varphi_1$}
			\put(63,0){$\Sigma_1$}
			\put(0,5){$\Sigma_2$}
			\put(90,35){$\varphi_2$}
			\put(45,3){$\varphi_1$}
			\put(16,32){$\varphi_2$}
			\put(70,75){$A_1$}
			\put(0,70){$A_2$}
			\put(20,0){$A_3$}
			\put(85,15){$A_4$}
		\end{overpic}
	\end{center}
	\caption{Illustration of an orbit of a hybrid system. The solid lines represent the orbit given by the flow of a smooth vector field. The dashed lines represent the instantaneous jump of a point $p$, given by a map $\varphi_i$, when it reaches the switching manifold $\Sigma=\Sigma_1\cup\Sigma_2$.}\label{Fig1}
\end{figure}
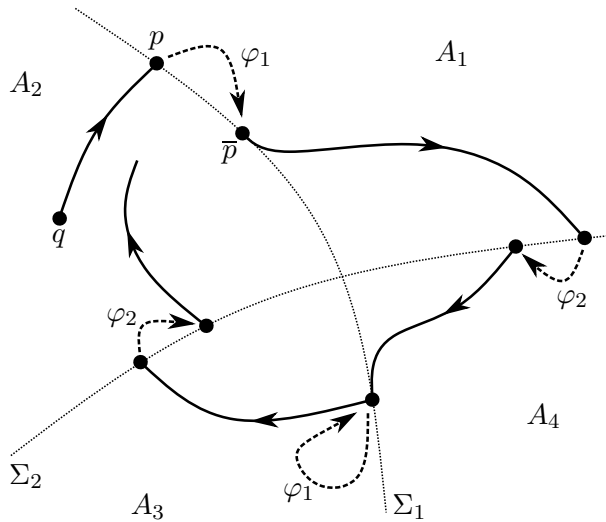
For instance, for a generalization of the stability methods of Lyapunov functions, we refer to Akhmet~\cite[Chapter $5$]{Akh2011}. In this paper we generalize the concepts of Cherkas~\cite{Cherkas} and Santana~\cite{San2023} about the stability of polycycles in smooth and Filippov systems to the hybrid framework. More precisely, suppose that we have a given planar hybrid system with a polycycle (such concepts will be properly defined later), see Figure~\ref{Fig2}.
\begin{figure}[ht]
	\begin{center}
		\begin{minipage}{8cm}	
			\begin{center}
				\begin{overpic}[width=7cm]{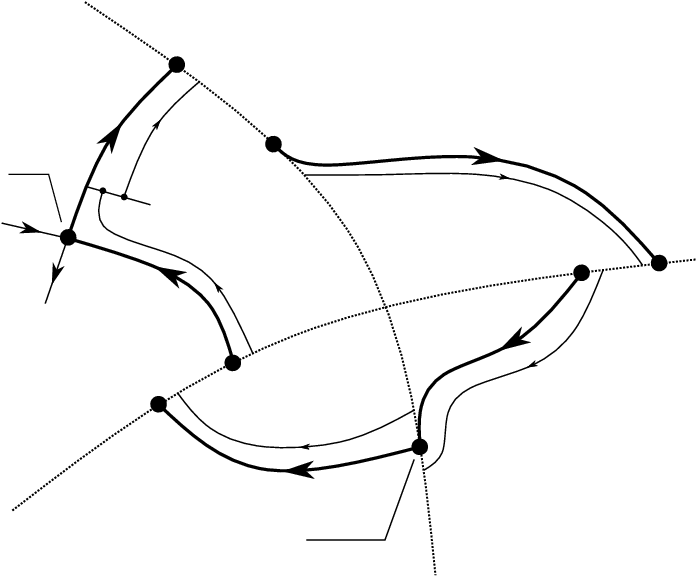}
				%\begin{overpic}[width=7cm,grid,tics=5]{Fig15.eps}
					\put(63,0){$\Sigma_1$}
					\put(0,5){$\Sigma_2$}
					\put(-4,57){$p_5$}
					\put(27.5,74){$p_4$}
					\put(41.5,62.5){$\overline{p_4}$}
					\put(95,41){$p_3$}
					\put(80,46){$\overline{p_3}$}
					\put(25,4){$p_2=\overline{p_2}$}
					\put(18,27){$p_1$}
					\put(34,26){$\overline{p_1}$}
					\put(19,39){$L_5$}
					\put(9,65){$L_4$}
					\put(75,61){$L_3$}
					\put(66,35){$L_2$}
					\put(33,11){$L_1$}
				\end{overpic}
			\end{center}
		\end{minipage}
		\begin{minipage}{8cm}	
			\begin{center}
				\begin{overpic}[width=7cm]{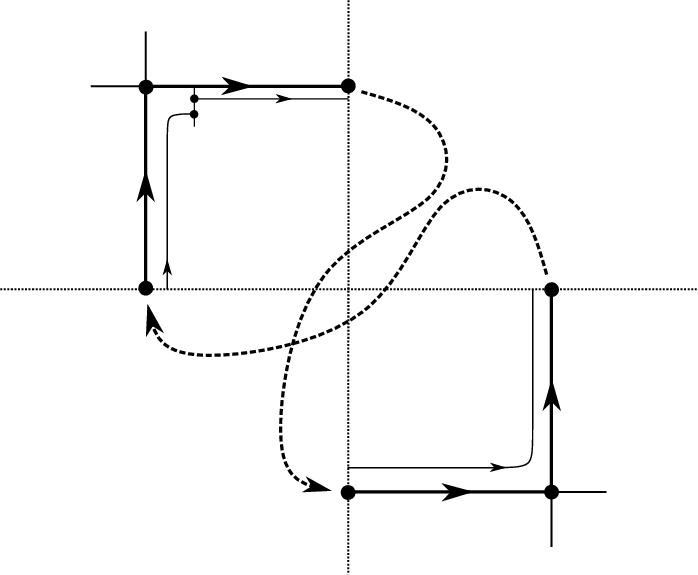}
				%\begin{overpic}[width=7cm,grid,tics=5]{Fig28.eps} 
					\put(15,72){$p_1$}
					\put(51,72){$p_2$}
					\put(51,7){$\overline{p_2}$}
					\put(81,8){$p_3$}
					\put(81,43){$p_4$}
					\put(14,36){$\overline{p_4}$}
					\put(98,42){$\Sigma_2$}
					\put(43,79){$\Sigma_1$}
					\put(33,20){$\varphi_1$}
					\put(25,27){$\varphi_2$}
					\put(13,50){$L_1$}
					\put(33,72){$L_2$}
					\put(65,6){$L_3$}
					\put(81,27){$L_4$}
				\end{overpic}
			\end{center}
		\end{minipage}
	\end{center}
\caption{Illustration of polycycles in hybrid systems. Here $\overline{p_i}$ denotes the image of $p_i$ after the jump. Observe that a jump may keep $p_i$ fixed while reversing the orientation, see $p_2$ on the left-hand side. In particular, polycycles can now be one-sided.}\label{Fig2}
\end{figure}
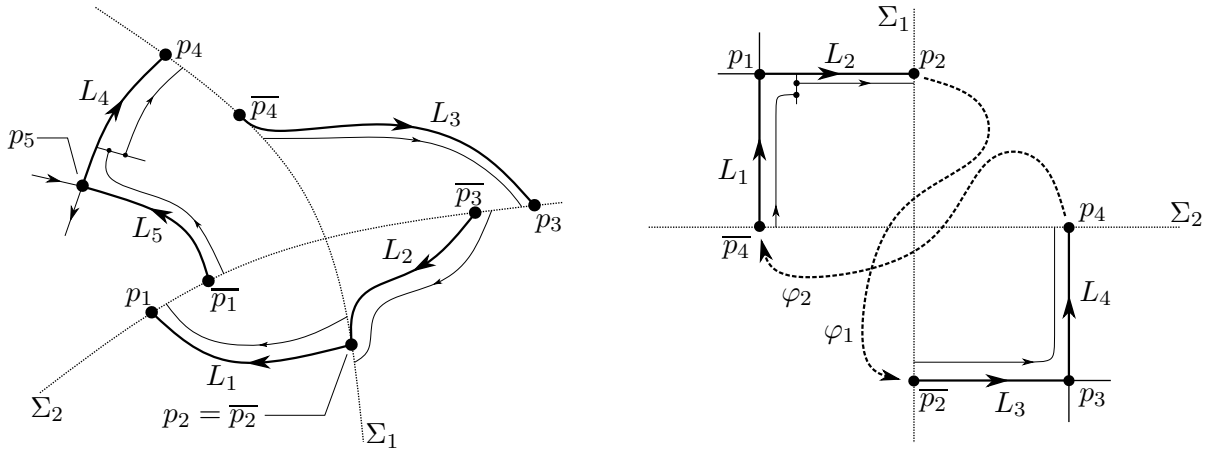
The goal of this paper is to approach such object and characterize its stability. There are some studies about polycycles in hybrid systems dealing, for example, with an extension of the Melnikov theory to study its cyclicity \cite{LiWeiZhaKap2023,LiMaZhaHao2016,LiWeiZhaYi1,WeiLiMorZha2}. Moreover, since Filippov systems can be seen as a particular case of a hybrid system (i.e. the case in which the jump is the identity map), simple hybrid systems are expected to have more limit cycles than their Filippov counterparts. For example, it follows from Llibre and Teixeira \cite{LliTei2018} that a Filippov system composed by two linear centers and such that the discontinuity is a straight line can not have limit cycles. However it follows from \cites{LiLiu2023,LliSan2024} that when we allow this Filippov system \emph{jump} at the discontinuity line we can have limit cycles. For a introduction about hybrid systems, we refer to the book of Schaft and Schumacher \cite{SchSch2000} and the paper of Branicky \cite{Bra2005}. For a recent survey in the field, we refer to Belykh, Kuske, Porfiri and Simpson~\cite{BKPS2023}.

As anticipated in the above discussion, in this paper we study the stability of polycycles in hybrid systems, extending the results of Cherkas~\cite{Cherkas} and Santana~\cite{San2023} about its smooth and Filippov counterparts to the hybrid case. More precisely let $X$ be a planar smooth vector field with a polycycle $\Gamma$, composed by $n$ singularities $p_1,\dots,p_n$, all of them being hyperbolic saddles. Let $\nu_i<0<\lambda_i$ be the associated eigenvalues of the saddle $p_i$, $i\in\{1,\dots,n\}$. The \emph{hyperbolicity ratio} of $p_i$ is the positive real number given by
\begin{equation}\label{2}
	r_i=\frac{|\nu_i|}{\lambda_i}.
\end{equation}
The \emph{graphic number} of $\Gamma$ is the positive real number given by
\begin{equation}\label{11}
	r=\prod_{i=1}^{n}r_i.
\end{equation}
Cherkas~\cite{Cherkas} proved that if $r>1$ (resp. $r<1$), then $\Gamma$ is stable (resp. unstable). In \cite{San2023} the author extended \eqref{2}, allowing the \emph{tangential singularities} of a given polycycle in a Filippov system also have a well defined hyperbolicity ratio. With this extension it was possible to extend the results of Cherkas to the Filippov case. In this paper we generalize the notion of hyperbolicity ratio to include a new kind of singularity that appears in hybrid systems, which we call \emph{jump singularities}. With this generalization we are now able to estabilish the notion of graphic number \eqref{11} and thus study the stability of a hybrid polycycle, which may contain hyperbolic saddles and tangential or jump singularities. 

The paper is organized as follows. In Section~\ref{Sec2} we have the precise statement of our main result. Section~\ref{Sec3} is devoted to organize the preliminaries results. The main result is proved in Section~\ref{Sec4}.

\section{Statement of the main results}\label{Sec2}

Before we establish the technical definitions, we present one of the usual examples of an hybrid dynamical system. The \emph{bouncing ball model}~\cite{LliSan2024}. 
\begin{example}[Example~$1$ of \cite{LliSan2024}]\label{Ex1}
	Consider the vertical motion of a ball dropped (or tossed) from an initial height $h>0$, with initial velocity $v\in\mathbb{R}$ (here negative velocity means downwards following gravity, while positive velocity means that the ball was tossed upwards). If the ball is under the acceleration of constant gravity $g>0$, then for $h>0$ the state of the ball is governed by the system of differential equations,
	\begin{equation}\label{18}
		\dot h=v, \quad \dot v=-g.
	\end{equation}
	Therefore given any initial condition $q=(h_0,v_0)$, $h_0>0$, it follows from \eqref{18} that the ball reaches the ground $h=0$ after a finite amount of time $t_0>0$, with velocity $v(t_0)<0$. See Figure~\ref{Fig5}.
	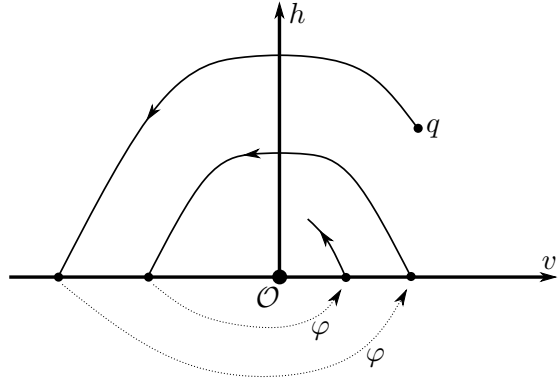
\begin{figure}[ht]
		\begin{center}
			\begin{overpic}[height=5cm]{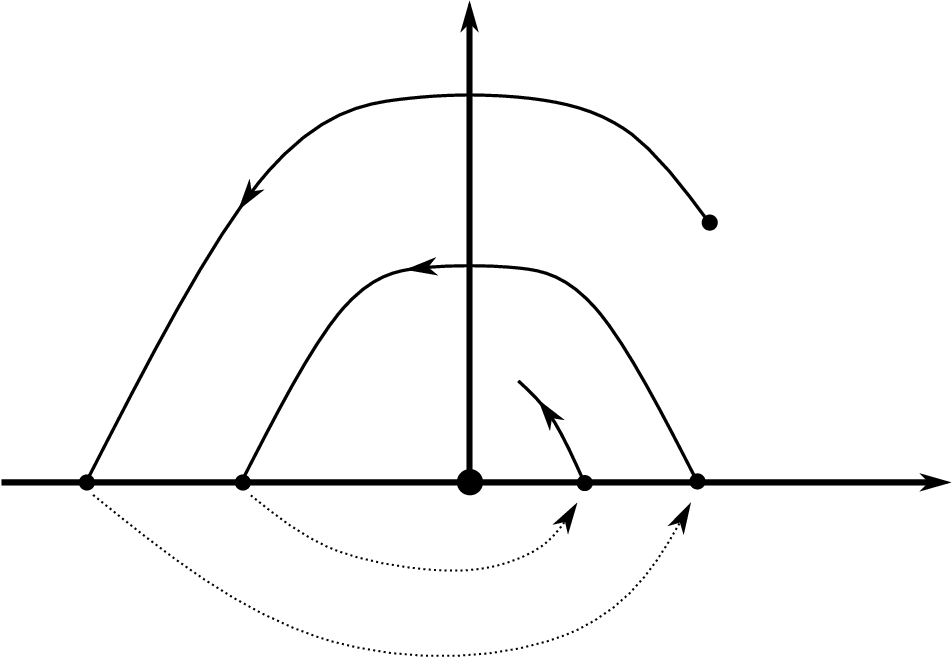} 
			%\begin{overpic}[height=5cm,grid,tics=5]{Fig22.eps} 
				\put(51,65){$h$}
				\put(97,20){$v$}
				\put(45,13){$\mathcal{O}$}
				\put(55,8){$\varphi$}
				\put(65,3){$\varphi$}
				\put(76,45){$q$}			
			\end{overpic}		
		\end{center}
		\caption{Illustration of an orbit of the bouncing ball model. For simplicity, we interchanged the coordinate axes.}\label{Fig12}
	\end{figure}
	At this instant the ball bounces back upwards and its velocity undergoes an instantaneous change $v(t_0)\mapsto -\rho v(t_0)$, modeled by an inelastic collision with dissipate factor $0<\rho<1$. That is, at $h=0$ our system undergoes a jump (or reset) $\varphi$ given by $\varphi(0,v)=(0,-\rho v)$. The reset map represents the instantaneous loss of energy that the ball suffers when hitting the ground. Observe that any orbit converges to the origin, which represents a ball standing still in the ground.
\end{example}
We now establish the necessary definitions and concepts for a proper statement of our main result. First, we start by defining what we mean by planar \emph{hybrid system}. 

Let $h_1,\dots,h_N\colon\mathbb{R}^2\to\mathbb{R}$, $N\geqslant 1$, be $C^\infty$-functions having $0$ as a regular value, i.e., such that $\nabla h_i(x,y)\neq(0,0)$ for every $(x,y)\in h^{-1}(\{0\})$, where $\nabla h$ denotes the gradient vector of $h$. For each $i\in\{1,\dots,N\}$, let $\Sigma_i=h_i^{-1}(\{0\})$ and observe that $\Sigma_i$ is a $1$-dimensional manifold. Let also $\Sigma=\cup_{i=1}^{N}\Sigma_i$. Let $A_1,\dots,A_M\subset\mathbb{R}^2$, $M\geqslant 2$, be the open connected components of $\mathbb{R}^2\backslash\Sigma$. For each $i\in\{1,\dots,M\}$, let $X_i$ be a $C^\infty$-planar vector field defined over $\overline{A_i}$, where $\overline{A_i}$ denotes the topological closure of $A_i$. For each $i\in\{1,\dots,N\}$, let $\varphi_i\colon\Sigma_i\to\Sigma_i$ be a $C^\infty$-map. To simplify the notation, let $\mathfrak{X}=\{X_1,\dots,X_M\}$ and $\Phi=\{\varphi_1,\dots,\varphi_N\}$. Let $Z=Z(\mathfrak{X},\Sigma,\Phi)$ be the associated \emph{hybrid system}. 

Before we state the way that the dynamics of $Z$ works, we need to establish some nomenclatures. First, we say that $X_1,\dots,X_M$ are the \emph{components} of $Z$ and that $\Sigma_1,\dots,\Sigma_N$ are the components of $\Sigma$. Let $p\in\Sigma$. We say that $p$ is \emph{regular} if the following statements hold.
\begin{enumerate}[label=(\alph*)]
	\item There is an unique $i_0\in\{1,\dots,N\}$ such that $p\in\Sigma_{i_0}$.
	\item For every $i\in\{1,\dots,N\}$, with $i\neq i_0$, we have $\varphi_{i_0}(p)\not\in\Sigma_i$.
\end{enumerate}
To simplify the notation let $\overline{p}=\varphi_{i_0}(p)$. It follows from statement $(a)$ that if $p\in\Sigma_{i_0}$ is regular, then there are exactly two connected components of $\mathbb{R}^2\backslash\Sigma$ having $p$ in its boundary. Moreover, $h_{i_0}$ is strictly positive in one of such components and strictly negative in the other. Let $A_p^+$ (resp. $A_p^-$) be the unique connected component of $\mathbb{R}^2\backslash\Sigma$ having $p$ in its boundary and such that $h_{i_0}(x,y)>0$ (resp. $h_{i_0}(x,y)<0$) for every $(x,y)\in A_p^+$ (resp. $(x,y)\in A_p^-$). Similarly let $A_{\overline{p}}^+$ and $A_{\overline{p}}^-$ be the analogous connected components associated to $\overline{p}$. We now define the dynamics of $Z$. Let $q\in\mathbb{R}^2\backslash\Sigma$ and let $A_{j_0}$ be the unique connected component of $\mathbb{R}^2\backslash\Sigma$ containing $q$ and let $X_{j_0}$ be the component of $Z$ defined over $\overline{A_{j_0}}$. The orbit of $Z$ through $q$ is given as follows. First we let $q$ follows the orbit of $X_{j_0}$. If such orbit never leaves $A_{j_0}$, then we are done, that is, the orbit of $Z$ is given by the orbit of $X_{j_0}$. Hence suppose that the orbit of $X_{j_0}$ through $q$ intersects $\Sigma$ in a point $p$. Suppose that $p\in\Sigma$ is regular and let $\overline{p}$, $A_p^+$, $A_p^-$, $A_{\overline{p}}^+$ and $A_{\overline{p}}^-$ be given as above. Observe that either $A_{j_0}=A_p^+$ or $A_{j_0}=A_p^-$. Without loss of generality suppose that $A_{j_0}=A_p^+$. After reaching $p$ the orbit \emph{jump} to $\overline{p}$. For simplicity let $A^-=A_{\overline{p}}^-$ and let $X^-$ be the component of $Z$ defined over $A^-$. Suppose now that the orbit $\gamma^-\colon[0,\varepsilon)\to\mathbb{R}^2$ of $X^-$ through $\overline{p}$ is well defined for some $\varepsilon>0$, with $\gamma^-(0)=\overline{p}$ and is such that $\gamma(t)\in A^-$ for every $t\in(0,\varepsilon)$. Thus after it jumps from $p$ to $\overline{p}$, we now let $\overline{p}$ follows the orbit of $X^-$ and thus we arrive at some $q_1\in A^-$. In particular we arrive at some $q_1\in\mathbb{R}^2\backslash\Sigma$. Hence we now repeat the process, i.e., when the orbit of $X^-$ through $q_1$ stays in $A^-$ then we are done. If it does not then it intersects $\Sigma$ in a point $p_1$ that we suppose to be regular and so on. See Figure~\ref{Fig1}.

We now briefly recall the notion of \emph{crossing} and \emph{tangential} points of Filippov systems, see Figure~\ref{Fig8}. 
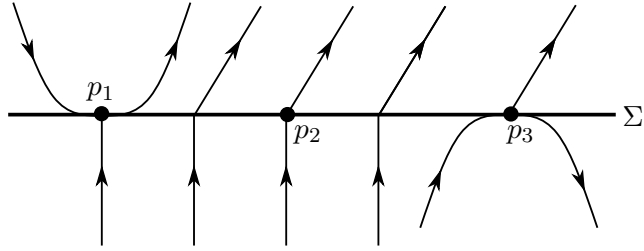
\begin{figure}[ht]
	\begin{center}
		\begin{overpic}[width=8cm]{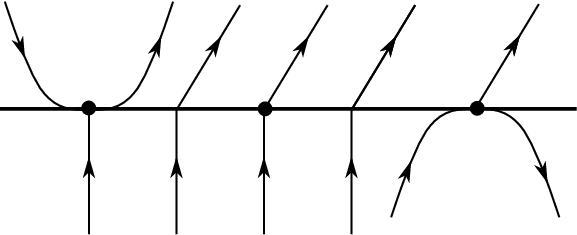}
		%\begin{overpic}[width=8cm,grid,tics=5]{Fig16.eps} 
			\put(13,25){$p_1$}
			\put(47,18){$p_2$}
			\put(82,18){$p_3$}
			\put(101,20){$\Sigma$}
		\end{overpic}
	\end{center}
	\caption{Illustration of a crossing $p_2$ and tangential points $p_1$ and $p_3$ of $Y$.}\label{Fig8}
\end{figure}
Let $Y=(Y^+,Y^-,\Sigma)$ be a planar Filippov system such $\Sigma=h^{-1}(0)$, for some smooth function $h\colon\mathbb{R}^2\to\mathbb{R}$, with $0$ being a regular value. Suppose also that $Y^+$ (resp. $Y^-$) is a smooth planar system defined in the region $h(x,y)\geqslant0$ (resp. $h(x,y)\leqslant0$). Let $p\in\Sigma$. The \emph{Lie derivative} of $h$ in the direction of $Y^+$ (resp. $Y^-$) at $p$ is given by
	\[Y^+h(p)=\left<Y^+(p),\nabla h(p)\right> \quad \bigl(\textnormal{resp. } Y^-h(p)=\left<Y^-(p),\nabla h(p)\right>\bigr),\]
where $\left<\cdot,\cdot\right>$ denotes the standard inner product of $\mathbb{R}^2$. In particular, observe that $Y^\pm h(p)=0$ if, and only if, $Y^\pm(p)$ is tangent to $\Sigma$ at $p$. Given $p\in\Sigma$, we say that $p$ is a \emph{crossing point} of $Y$ if
	\[Y^+h(p)Y^-h(p)>0.\]
Moreover, we say that $p$ is a \emph{tangential singularity} of $Y$ if
	\[Y^+(p)\neq0, \quad Y^-(p)\neq0 \quad \textnormal{and}  \quad Y^+h(p)Y^-h(p)=0.\]
Knowing that a Filippov system can be seen as a particular case of a hybrid system whose jumps are given by the identity map, we now propose the following generalization of crossing points and tangential singularities, see Figure~\ref{Fig9}.
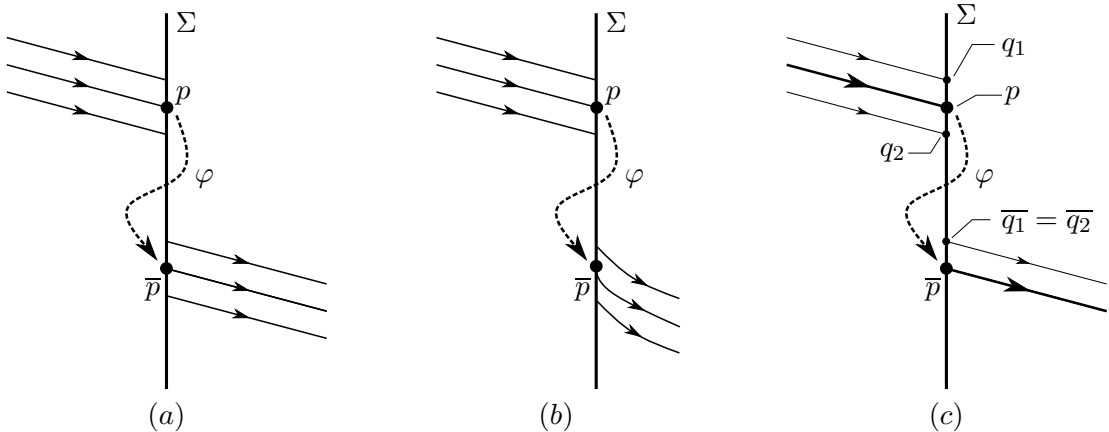
\begin{figure}[ht]
	\begin{center}
		\begin{minipage}{5cm}
			\begin{center} 
				\begin{overpic}[height=5cm]{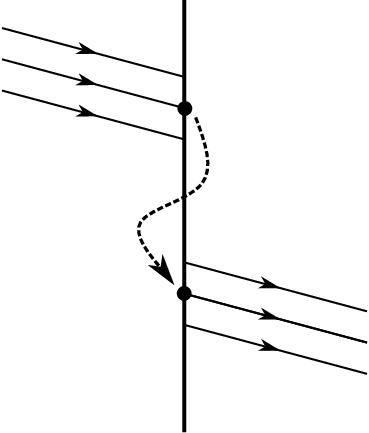} 
				%\begin{overpic}[height=5cm,grid,tics=5]{Fig17.eps} 
					\put(45,77){$p$}
					\put(37,25){$\overline{p}$}
					\put(50,55){$\varphi$}
					\put(45,95){$\Sigma$}
				\end{overpic}
				
				$(a)$
			\end{center}
		\end{minipage}
		\begin{minipage}{5cm}
			\begin{center} 
				\begin{overpic}[height=5cm]{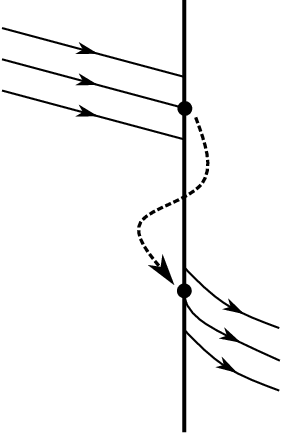} 
				%\begin{overpic}[height=5cm,grid,tics=5]{Fig18.eps} 
					\put(45,77){$p$}
					\put(37,25){$\overline{p}$}
					\put(50,55){$\varphi$}
					\put(45,95){$\Sigma$}
				\end{overpic}
				
				$(b)$
			\end{center}
		\end{minipage}
		\begin{minipage}{5cm}
			\begin{center} 
				\begin{overpic}[height=5cm]{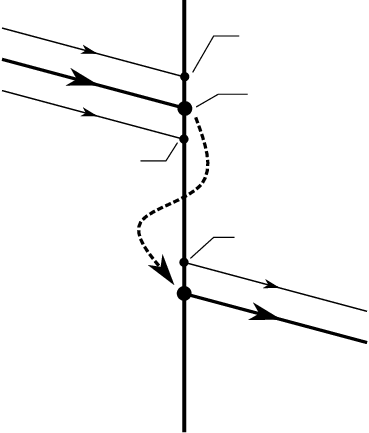} 
				%\begin{overpic}[height=5cm,grid,tics=5]{Fig21.eps} 
					\put(58,77){$p$}
					\put(37,25){$\overline{p}$}
					\put(57,43){$\overline{q_1}=\overline{q_2}$}
					\put(57,90){$q_1$}
					\put(25,62){$q_2$}
					\put(50,55){$\varphi$}
					\put(45,97){$\Sigma$}
				\end{overpic}
			
				$(c)$
			\end{center}
		\end{minipage}
	\end{center}
	\caption{Illustration of a $(a)$ jump crossing point, a $(b)$ jump singularity satisfying $(i)$ and $(c)$ a jump singularity satisfying $(ii)$.}\label{Fig9}
\end{figure}

\begin{definition}\label{definicao1}
	Let $Z=(\mathfrak{X},\Sigma,\Phi)$ be a hybrid system and $(p,\overline{p})\in\Sigma_i$ a regular point. Let $X_p^+$, $X_p^-\in\mathfrak{X}$ be the two components of $Z$ defined at $p$. Let also $\overline{p}=\varphi_i(p)$ and $X_{\overline{p}}^+$, $X_{\overline{p}}^-$ be the components of $Z$ defined at $\overline{p}$. We say that the pair $(p,\overline{p})$ is a \emph{jump crossing point} of $Z$ in $\varphi_i$ if there is a neighborhood $N\subset\Sigma_i$ of $p$ such that $\varphi|_N$ is a diffeomorphism and
	\begin{equation}\label{1}
		X^{+}_ph_i(p)X_{\overline{p}}^-h_i(\overline{p})>0.
	\end{equation}
	Moreover, we say that $(p,\overline{p})$ is a \emph{jump singularity} if $X_p^+(p)\neq0$, $X_{\overline{p}}^-(\overline{p})\neq0$ and if at least one of the following statements hold.
	\begin{enumerate}[label=(\roman*)]
		\item $X_p^+h_i(p)X_{\overline{p}}^-h_i(\overline{p})=0$.
		\item For every neighborhood $N\subset\Sigma_i$ of $p$, $\varphi|_N$ is not a diffeomorphism.
	\end{enumerate}
\end{definition}

\begin{remark}
	In the definition \ref{definicao1}, in the inequality \eqref{1} of jump crossing point, we are implicitly assuming that $p$ arrives at the switching manifold $\Sigma_i$ through the flow of $X_p^+$, while $\overline{p}$ leaves it through the flow of $X_{\overline{p}}^-$. If the opposite happens, then we replace \eqref{1} by $X^{-}_ph_i(p)X_{\overline{p}}^+h_i(\overline{p})>0$. A similar observation holds for the definition of jump singularity.
	
	For simplicity, throughout the paper giving a jump crossing point or jump singularity $(p,\overline{p})$, we always assume that $p$ arrives at the switching manifold $\Sigma_i$ through the flow of $X_p^+$, while $\overline{p}$ leaves it through the flow of $X_{\overline{p}}^-$.
\end{remark}

\begin{definition}
	Consider a point $q\in\Sigma$, $X$ one of the components of $Z$ defined at $q$ and let $X^kh_i(q)=\left<X(q),\nabla X^{k-1}h_i(q)\right>$, $k\geqslant2$. We say that $X$ has $m$-order contact with $\Sigma$ at $q$, $m\geqslant1$, if $m$ is the first positive integer such that $X^mh_i(q)\neq0$. If $X^kh_i(q)=0$ for every $k\in\mathbb{N}$, then we say that $q$ is a \emph{flat point} of $X$.
\end{definition}

\begin{definition}[Definition of Polycycle]
	Let $Z=(\mathfrak{X},\Sigma,\Phi)$ be a hybrid system. A \emph{graphic} of $Z$ is a compact, non-empty invariant subset composed by singularities $(p_1,\overline{p_1}),\dots,(p_n,\overline{p_n}),(p_{n+1},\overline{p_{n+1}})=(p_1,\overline{p_1})$ (not necessarily distinct), with each $(p_i,\overline{p_i})$ being either a jump singularity of $Z$ or a singularity of some component of $Z$ (in this case, $p_i=\overline{p_i}$); and distinct regular orbits $L_1,\dots,L_n$ such that $\overline{p_{i+1}}$ and $p_i$ are the $\alpha$ and $\omega$-limits of $L_i$.
	
	A \emph{polycycle} is a graphic with a well defined first return map $\pi\colon\ell\to\ell$. A polycycle is \emph{hyperbolic} if it satisfies the following statements.
	\begin{enumerate}[label=(\alph*)]
		\item Except perhaps by $\overline{p_{i+1}}$ and $p_i$, each orbit $L_i$ intersects $\Sigma$ at most in a finite number of points, with each of them being a jump crossing point.
		\item Each singularity $(p_i,\overline{p_i})$ satisfies exactly one of the following conditions:
		\begin{enumerate}[label=(\roman*)]
			\item $p_i=\overline{p_i}$ is a hyperbolic saddle and $p_i\not\in\Sigma$.
			\item $(p_i,\overline{p_i})$ is a jump singularity and $p_i$ (resp. $\overline{p_i}$) is not a flat point of $X_p^+$ (resp. $X_{\overline{p}}^-$). 
		\end{enumerate}
	\end{enumerate}
\end{definition}

From now on, we let $\Gamma^n$ denote a hyperbolic polycycle with $n$ singularities $(p_1,\overline{p_1}),\dots,(p_n,\overline{p_n})$ (not necessarily distinct) and $n$ distinct regular orbits $L_1,\dots, L_n$ such that $L_i$ is an orbit from $\overline{p_{i+1}}$ to $p_i$. See Figure~\ref{Fig10}.
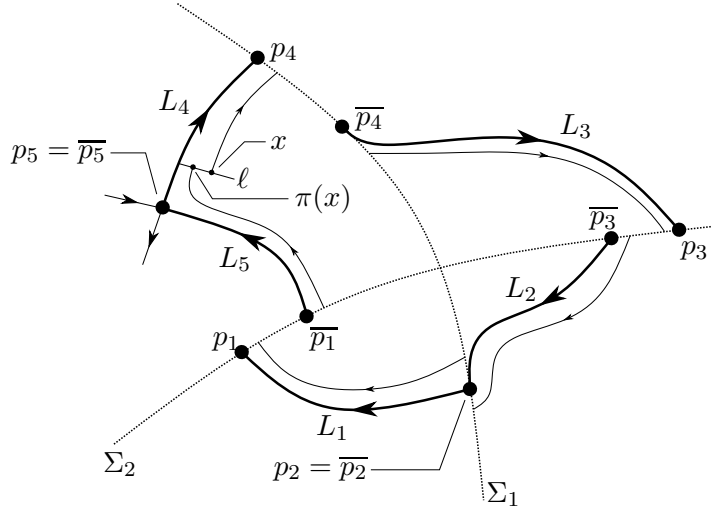
\begin{figure}[ht]
	\begin{center}
		\begin{overpic}[width=8cm]{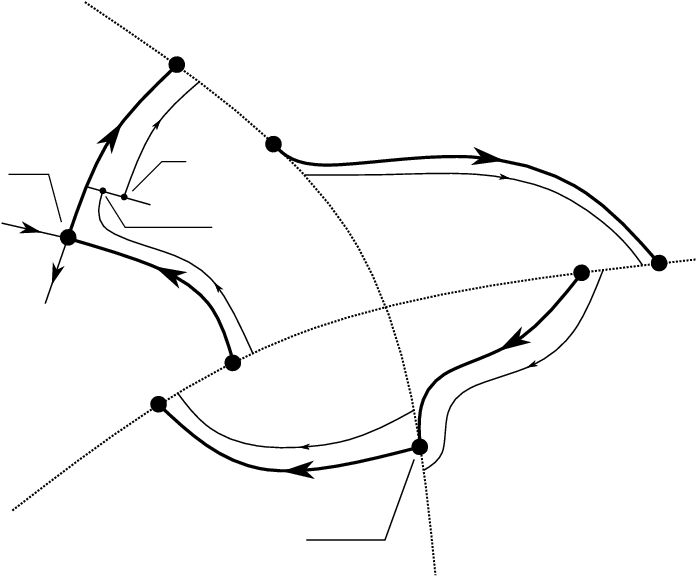}
		%\begin{overpic}[width=8cm,grid,tics=5]{Fig2.eps} %Esse é só para colocar os \put
			\put(63,0){$\Sigma_1$}
			\put(0,5){$\Sigma_2$}
			\put(-15,57){$p_5=\overline{p_5}$}
			\put(27.5,74){$p_4$}
			\put(41.5,62.5){$\overline{p_4}$}
			\put(95,41){$p_3$}
			\put(80,46){$\overline{p_3}$}
			\put(28,4.5){$p_2=\overline{p_2}$}
			\put(18,26){$p_1$}
			\put(34,26){$\overline{p_1}$}
			\put(19,39){$L_5$}
			\put(9,65){$L_4$}
			\put(75,61){$L_3$}
			\put(66,34){$L_2$}
			\put(35,11){$L_1$}
			\put(27.5,58.5){$x$}
			\put(31.5,49){$\pi(x)$}
			\put(22,52){$\ell$}
		\end{overpic}
	\end{center}
\caption{An example of a hyperbolic polycycle $\Gamma^5$. Observe that although we have $\varphi_1(p_2)=p_2$, the map $\varphi_1$ reversed the orientation of the returning map at $p_2$. In particular $p_2$ is a jump singularity but it is not a tangential singularity.}\label{Fig10}
\end{figure}

Given a jump singularity $(p,\overline{p})$ of $\Gamma^n$, let $L_s$, $L_u\in\{L_1,\dots,L_n\}$ be the orbits of $\Gamma^n$ such that $p$ is the $\omega$-limit of $L_s$ and $\overline{p}$ the $\alpha$-limit of $L_u$, see Figure~\ref{Fig11}. The \emph{stable} and \emph{unstable} components of $(p,\overline{p})$ are the components $X_s$, $X_u\in\{X_1,\dots,X_M\}$ of $Z$ defined over $L_s$ and $L_u$ respectively. 
\begin{figure}[ht]
	\begin{center}
		\begin{minipage}{6cm}
			\begin{center} 
				\begin{overpic}[height=5cm]{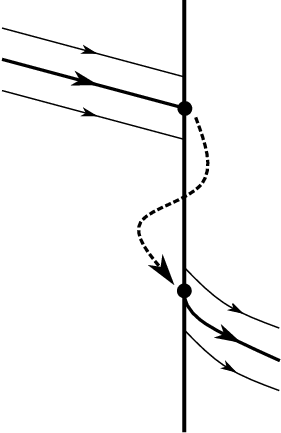} 
					%\begin{overpic}[height=5cm,grid,tics=5]{Fig19.eps} 
					\put(45,77){$p$}
					\put(37,25){$\overline{p}$}
					\put(50,55){$\varphi$}
					\put(45,95){$\Sigma$}
					\put(-9,84){$L_s$}
					\put(66,14){$L_u$}
				\end{overpic}
				
				$n_s=1$, $n_u=2k+1,k\in\mathbb{Z}$.
			\end{center}
		\end{minipage}
		\begin{minipage}{6cm}
			\begin{center} 
				\begin{overpic}[height=5cm]{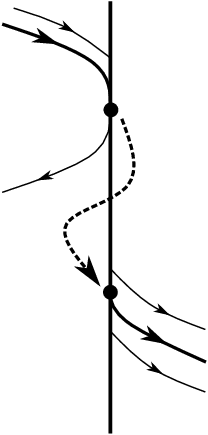} 
					%\begin{overpic}[height=5cm,grid,tics=5]{Fig20.eps} 
					\put(28,77){$p$}
					\put(19,25){$\overline{p}$}
					\put(32,55){$\varphi$}
					\put(28,95){$\Sigma$}					
					\put(-8,92){$L_s$}
					\put(48,14){$L_u$}
				\end{overpic}
				
				$n_s=2r$, $n_u=2s+1,r,s\in\mathbb{Z}$.
			\end{center}
		\end{minipage}
	\end{center}
	\caption{Illustration of $L_s$, $L_u$ and their different types of contact orders.}\label{Fig11}
\end{figure}
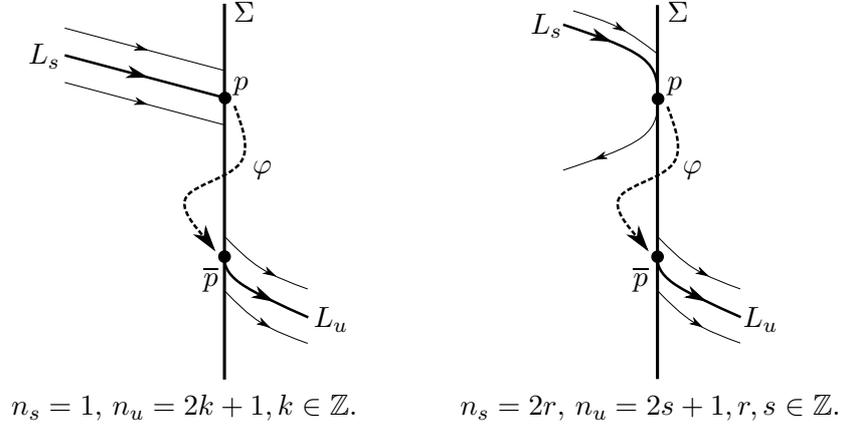

\begin{definition}
	Given a jump singularity $(p,\overline{p})$, let $X_s$ and $X_u$ be its stable and unstable components. We define the stable and unstable contact order of $p$ as the contact order $n_s$ and $n_u$ of $X_s$ and $X_u$ with $\Sigma$ at $p$ and $\overline{p}$ respectively. Moreover, if $n_u$ and $n_s$ are finite then we say that $p$ is a non-flat jump singularity of $Z$.
\end{definition}

It follows from \cite{San2023} that the contact orders $n_s$ and $n_u$ of a tangential singularity play the same role as the eigenvalues $\nu<0<\lambda$ of a hyperbolic saddle. In particular, the fraction $n_u/n_s$ play the same role as the hyperbolicity ratio~\eqref{2}. That is, $n_u/n_s$ characterizes when a tangential singularity contracts or expands the flow around it. In case of hybrid systems, we must take the action of the jump $\varphi$ into account, as it can also be contracting or expanding. We now address this.

Let $(p,\overline{p})\in\Sigma_i$ be a jump singularity of $Z$ and let $I$, $J\subset\Sigma_i$ be neighborhoods of $p$ and $\overline{p}$ respectively (we recall that $\Sigma_i$ is an one-dimensional manifold). After taking charts if necessary, we can identify $I=(-\delta,\delta)$ and $J=(-\varepsilon,\varepsilon)$ with $\delta>0$ and $\varepsilon>0$ small enough in addition to $p=0$ and $\overline{p}=0$. Hence we can write the $C^\infty$-map $\varphi_i\colon I\to J$ as
	\[\varphi_i(t)=\left(\sum_{k=1}^{\infty}a_kt^k\right)+R(t),\]
with $a_k\in\mathbb{R}$ and $R\colon I\to J$ a flat $C^\infty$-function such that $R(0)=0$. If $a_k\neq0$ for some $k\in\mathbb{N}$, then we say that $p$ is a jump singularity of \emph{finite power order}. Moreover, in this case if $k_0\in\mathbb{N}$ is the first integer such that $a_{k_0}\neq0$ then we say that $p$ has power of order $k_0$ or that $k_0$ is the power order of $p$. On the other hand if $a_k=0$ for every $k\in\mathbb{N}$, then we say that $p$ has \emph{infinity power order}.

\begin{definition}
	Let $\Gamma^n$ be a hyperbolic polycycle. The hyperbolicity ratio $r_i\in\mathbb{R}_{>0}\cup\{\infty\}$ of $p_i$ is defined as follows.
	\begin{enumerate}[label=(\alph*)]
		\item If $p_i$ is a jump singularity of finite power order then $r_i=\frac{n_{i,u}}{n_{i,s}}k_i$, where $n_{i,s}$ and $n_{i,u}$ are the stable and unstable contact orders of $p_i$ and $k_i$ is the power order of $p$.
		\item If $p_i$ is a jump singularity of infinity power order, then $r_i=\infty$.
		\item If $p_i$ is a hyperbolic saddle, then $r_i=\frac{|\nu_i|}{\lambda_i}$, where $\nu_i<0<\lambda_i$ are the eigenvalues of the jacobian matrix of $Z$ at $p$ (recall that in this case $p\not\in\Sigma)$.
	\end{enumerate}
	Moreover, the \emph{graphic number} $r\in\mathbb{R}_{>0}\cup\{\infty\}$ of $\Gamma^n$ is given by
		\[r=\left\{\begin{array}{ll}
			\displaystyle\prod_{i=1}^{n}r_i, & \text{if every jump singularity has finite power order,} \\
			\displaystyle\infty, & \text{otherwise}.
		\end{array}\right.\]
\end{definition}

Let $\Gamma^n$ be a hyperbolic polycycle and $\pi\colon\ell\to\ell$ be its associated first return map, see Figure~\ref{Fig10}. Without loss of generality we can identify $\ell$ with the semi-open interval $I=\{x\in\mathbb{R}\colon 0\leqslant x<\delta\}$, with $\delta>0$ small enough such that the intersection of $\ell$ with $\Gamma^n$ is given by $x=0$. We say that $\Gamma^n$ is \emph{stable} (resp. \emph{unstable}) if $\pi(x)<x$ (resp. $\pi(x)>x$) for every $0<x<\delta$. Our main result is the following.

\begin{theorem}\label{T1}
	Let $Z$ be a hybrid system with a hyperbolic polycycle $\Gamma^n$ with graphic number $r$. Then $\Gamma^n$ is stable if $r>1$ and unstable if $r<1$. In particular if $\Gamma^n$ has a jump singularity with infinity power order, then $\Gamma^n$ is stable.
\end{theorem}

We observe that if we suppose that $\varphi_i=\text{Id}_{\Sigma_i}$ for every $i\in\{1,\dots,N\}$, then the hybrid system $Z$ becomes a Filippov system. In this case our main result reduces to \cite[Theorem~$1$]{San2023}, which deals with the stability of polycycles in Filippov systems. Moreover, if we also suppose that every singularity of $\Gamma^n$ is of jump type (and consequently tangential) and $n_{i,s}=1$ for every $i\in\{1,\dots,n\}$ (where $n_{i,s}$ is the stable contact order of $p_i$), then we have the so called \emph{regular-tangential} $\Sigma$-polycycles studied by Andrade, Gomide and Novaes~\cite{AndGomNov2023}. In this context the authors in \cite[Propositions~$1$ and $3$]{AndGomNov2023} proved that such polycycles are stable. We observe that this fact can also be obtained from our main result, by observing that $r_i>1$ for every $i\in\{1,\dots,n\}$ and thus $r>1$. Finally, we observe that if we suppose $\varphi_i=\text{Id}_{\Sigma_i}$ for every $i\in\{1,\dots,N\}$ and $X_j=X$ for every $j\in\{1,\dots,M\}$, then $Z=X$ becomes a planar smooth vector field. In this case Theorem~\ref{T1} becomes the main result of Cherkas \cite{Cherkas}, which deals with the stability of generic polycycles in planar smooth vector fields, as described in Section~\ref{Sec1}.

\section{Preliminaries}\label{Sec3}

\subsection{Transition map near a hyperbolic saddle}\label{Sec3.1}

Let $X$ be a planar vector field of class $C^\infty$ such that the origin $\mathcal{O}$ is a hyperbolic saddle. After a $C^\infty$-change of variables, we can assume that the stable and unstable manifolds of $\mathcal{O}$ are given by the axis $Oy$ and $Ox$ respectively in a sufficient small neighborhood of $\mathcal{O}$. Let
	\[\sigma=\{(x,y_0)\in\mathbb{R}^2\colon0<x<x_1\}, \quad \tau=\{(x_0,y)\in\mathbb{R}^2\colon 0<y<y_1\},\]
where $y_0>0$, $x_0>0$, $y_1>0$ and $x_1>0$ are small enough. The \emph{Dulac map} $D\colon\sigma\to\tau$ is the map given by the flow of $X$. See Figure~\ref{Fig5}.
\begin{figure}[ht]
	\begin{center}
		\begin{overpic}[height=4cm]{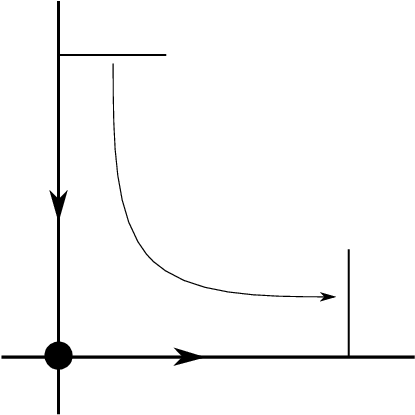} 
			\put(42,85){$\sigma$}
			\put(82,42){$\tau$}
			\put(4,4){$\mathcal{O}$}
			\put(0,94){$Oy$}
			\put(89,5){$Ox$}
			\put(38,42){$D$}
		\end{overpic}
	\end{center}
	\caption{The Dulac map near a hyperbolic saddle.}\label{Fig5}
\end{figure}
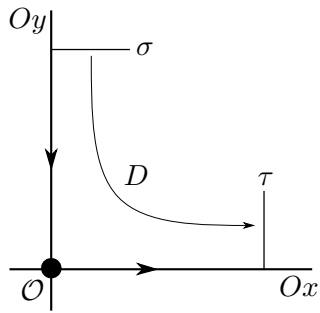

\begin{proposition}[Lemma $2.3$, Section $2.2$ of \cite{Soto}]\label{P1}
	Let $X$ be a vector field of class $C^\infty$ with a hyperbolic saddle at the origin $\mathcal{O}$ with the associated eigenvalues $\nu<0<\lambda$. Suppose also that the stable and unstable manifold of $p$ are given by the axis $Oy$ and $Ox$ respectively and let $D\colon\sigma\to\tau$ be the associated Dulac map. Then given $\varepsilon>0$, there is $\delta>0$ such that
		\[\delta^{1-\frac{|\nu|}{\lambda-\varepsilon}}x^{\frac{|\nu|}{\lambda-\varepsilon}}<D(x)<\delta^{1-\frac{|\nu|}{\lambda+\varepsilon}}x^{\frac{|\nu|}{\lambda+\varepsilon}}.\]	
\end{proposition}

We observe that Proposition~\ref{P1} is due to Sotomayor~\cite{Soto}. Since the original reference is not in English and, as far as we known there was no translation of it, the authors provided a translation of the original proof. See \cite[Proposition $2$]{San2023}. For a complete characterization of the Dulac map we refer to the works of Marin and Villadelprat~\cite{MarVil2020,MarVil2021,MarVil2024}.

\subsection{Transition map near a jump singularity}\label{Sec3.2}

Let $Z=(\mathfrak{X},\Sigma,\Phi)$ be a hybrid system with a hyperbolic polycycle $\Gamma^n$. Let $(p,\overline{p})\in\Sigma_i$ be a non-flat jump singularity of $\Gamma^n$ and $X_s$, $X_u$ be the stable and unstable components of $Z$ defined at $p$. Let also $n_s$ and $n_u$ be the stable and unstable contact orders of $p$. The aim of this section is to create a \emph{transition map} that works near the jump singularity $(p,\overline{p})$. To do this, first we assume that the map $\varphi_i\colon\Sigma_i\to\Sigma_i$ is such that there is a neighborhood $J\subset\Sigma_i$ of $p$ with $\varphi_i\colon J\to\Sigma_i$ the identity map. In particular $\overline{p}=p$ is now a tangential singularity and in a neighborhood $W\subset\mathbb{R}^2$ of $p$, the hybrid system $Z$ becomes a Filippov system. In this case let $B$ be a small enough neighborhood of $p$ and $\Psi:B\to\mathbb{R}^2$ be a $C^\infty$-change of coordinates such that $\Psi(p)=0$ and $\Psi(B\cap\Sigma)=Ox$. Let $l_s=\Psi(B\cap L_s)$, $l_u=\Psi(B\cap L_u)$ and $\tau_s$, $\tau_u$ be two small enough cross sections of $l_s$ and $l_u$ respectively. Let also
	\[\sigma=[0,\varepsilon)\times\{0\}, \quad \sigma=(-\varepsilon,0]\times\{0\}, \text{ or } \sigma=\{0\}\times[0,\varepsilon),\]
depending on $\Gamma^n$. It follows from \cite[Theorem~$1$]{AndGomNov2023} that $\Psi$ can be choosen such that the transition maps $T^{s,u}\colon\sigma\to\tau_{s,u}$ given by the flow of $X_{s,u}$ in this new coordinate system are well defined and are written as
\begin{equation}\label{3}
	T^u(x)=k_ux^{n_u}+O(x^{n_u+1}), \quad T^s(x)=k_sx^{n_s}+O(x^{n_s+1}),
\end{equation}
with $k_u\neq0$ and $k_s\neq0$. See Figure~\ref{Fig4}.
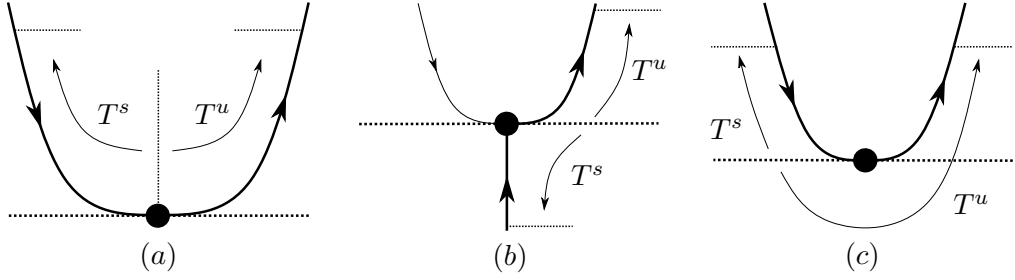
\begin{figure}[ht]
	\begin{center}
		\begin{minipage}{4.5cm}
			\begin{center}
				\begin{overpic}[width=4cm]{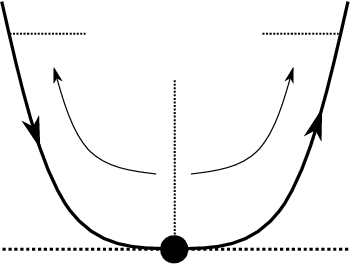} 
					\put(62,35){$T^u$}
					\put(30,35){$T^s$}
				\end{overpic}
				
				$(a)$
			\end{center}
		\end{minipage}
		\begin{minipage}{4.5cm}
			\begin{center}
				\begin{overpic}[width=4cm]{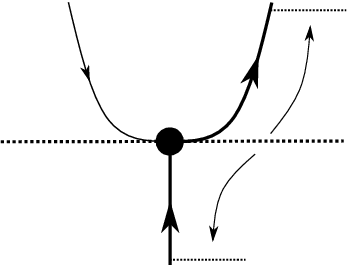} 
					\put(90,50){$T^u$}
					\put(70,15){$T^s$}
				\end{overpic}
				
				$(b)$
			\end{center}
		\end{minipage}
		\begin{minipage}{4.5cm}
			\begin{center}
				\begin{overpic}[width=4cm]{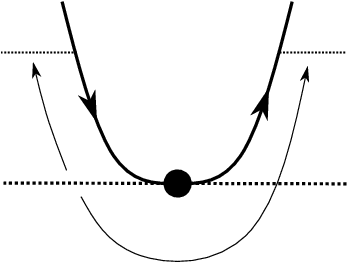} 
					\put(80,5){$T^u$}
					\put(0,30){$T^s$}
				\end{overpic}
		
				$(c)$
			\end{center}
	\end{minipage}
	\end{center}
\caption{Illustration of the maps $T^u$ and $T^s$.}\label{Fig4}
\end{figure}
We now work on the general case of a hybrid system $Z$. Let $(p,\overline{p})\in\Sigma_i$ be a jump of $\Gamma^n$ and let $X_s$ be the stable component of $p$ defined over the connected component $A_s$. Let $W\subset\mathbb{R}^2$ be a small neighborhood of $p$ and $\eta$ be a unitary normal vector to $\Sigma_i$ at $p$. Restricting $W$ if necessary, we can assume that $\eta$ is transversal to $\Sigma\cap W$. Let $A=W\backslash\overline{A_s}$ be the other connected component of $Z|_W$. On $W=A_s\cup\Sigma\cup A$ we construct the Filippov system $Y$ given by
	\[Y(x,y)=\left\{\begin{array}{ll}
					X_s(x,y), & \text{ if } (x,y)\in A_s, \vspace{0.2cm} \\
					\eta, & \text{ if } (x,y)\in A.
			  \end{array}\right.\]
Replacing $\eta$ by $-\eta$ if necessary, we observe that $Y$ has an unique and well defined regular orbit $L_u\subset A$ having $p$ as $\alpha$-limit. Hence it follows from \cite[Theorem~$1$]{AndGomNov2023} that after a change of variables the transitions maps $T^s$ and $T^u$ are well defined for $Y$ and are given by \eqref{3}. Since $Y|_{W\cap \overline{A_s}}=Z|_{W\cap\overline{A_s}}$, it follows that $T^s$ provides a transition map $T_p^s$ for $Z$ near $p$. Similarly it follows that we can obtain a transition map $T_{\overline{p}}^u$ of $Z$ near $\overline{p}$. See Figure~\ref{Fig3}.	
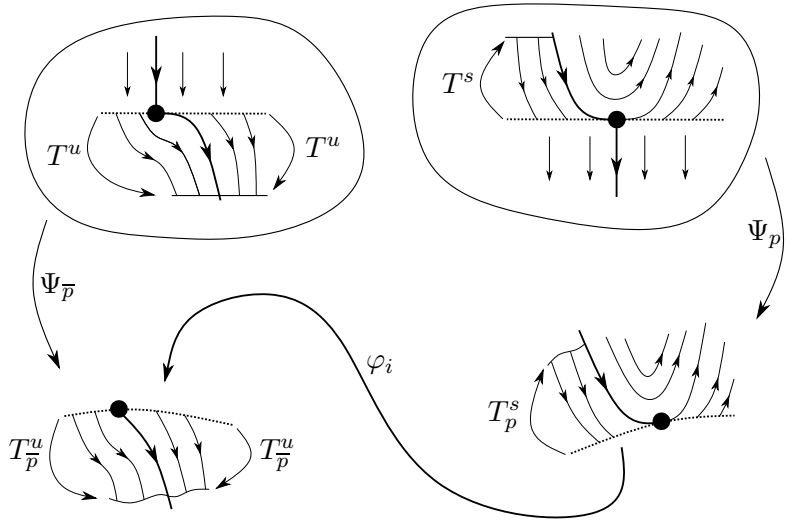
\begin{figure}[ht]
	\begin{center}
		\begin{overpic}[width=10cm]{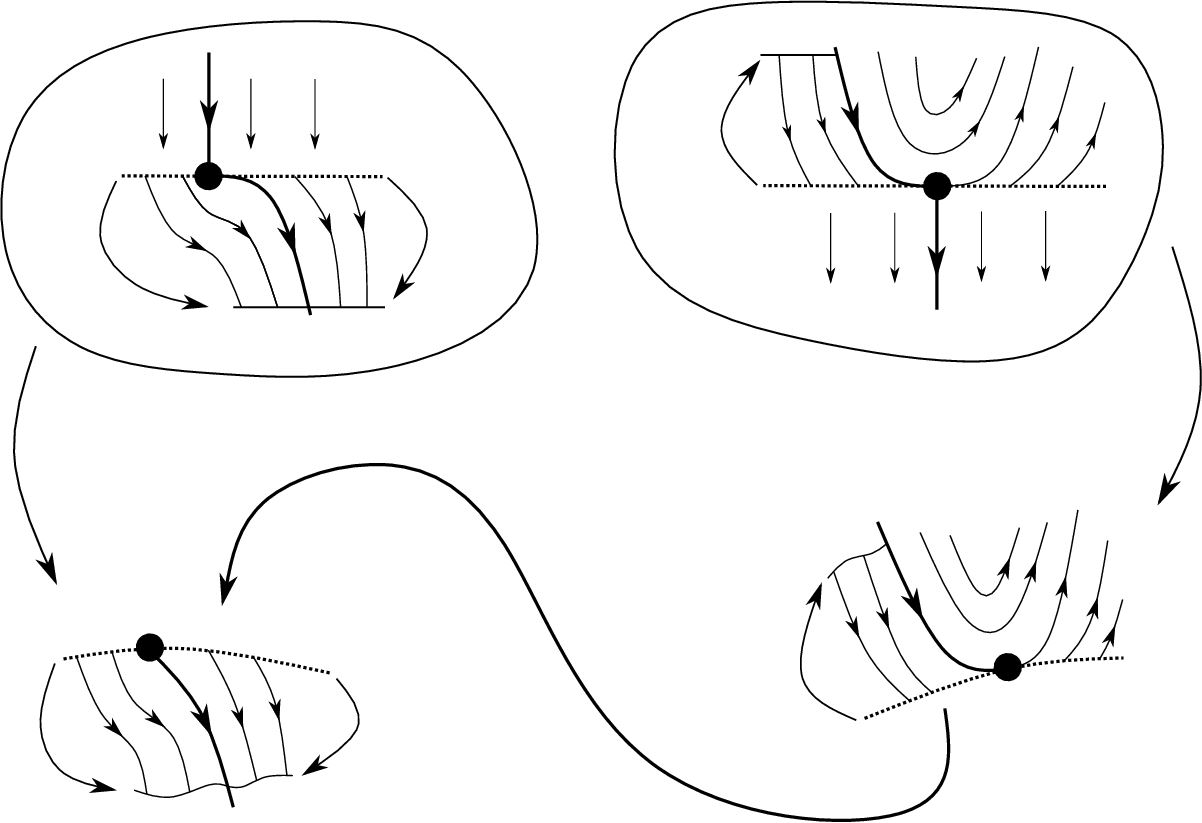}
		%\begin{overpic}[width=10cm,grid,tics=5]{Fig3.eps} %Esse é só para colocar os \put
			\put(55,56.5){$T^s$}
			\put(37,48){$T^u$}
			\put(3,47){$T^u$}
			\put(61,13){$T_p^s$}
			\put(31,8){$T_{\overline{p}}^u$}
			\put(-2,8){$T_{\overline{p}}^u$}
			\put(45,20){$\varphi_i$}
			\put(95,37){$\Psi_p$}
			\put(2,30){$\Psi_{\overline{p}}$}
		\end{overpic}
	\end{center}
	\caption{Illustration of the process to obtain the transition maps $T_p^s$ and $T_{\overline{p}}^u$.}\label{Fig3}
\end{figure}

\section{Proof of the main result}\label{Sec4}

We now proof our main result. For simplicity we assume that $Z=(\mathfrak{X},\Sigma,\Phi)$ is such that $Z$ has two components $X_1$, $X_2$ and $\Sigma=h^{-1}(0)$ has one component. We also suppose that $\Gamma^n=\Gamma^3$ is composed by a hyperbolic saddle $p_3$ and two jump singularities $(p_2,\overline{p_2})$ and $(p_1,\overline{p_1})$, with $\overline{p_1}=p_1$, see Figure~\ref{Fig6}. As we shall see, the proof of the general case follows similarly.

\begin{proof}[Proof of Theorem~\ref{T1}] Let $B_3$, $B_2$, $B_{\overline{2}}$ and $B_1$ be small enough neighborhoods of $p_3$, $p_2$, $\overline{p_2}$ and $p_1$ respectively and let $\Psi_3\colon B_3\to\mathbb{R}^2$ be the change of variables given by Section~\ref{Sec3.1}. Let $X_s$, $X_u$ be the stable and unstable components of $p_2$ defined over the connected components $A_s$, $A_u\subset\mathbb{R}^2\backslash\Sigma$. Let $\Psi_2\colon B_2\cap\overline{A_s}\to\mathbb{R}^2$ and $\Psi_{\overline{2}}\colon B_2\cap\overline{A_u}\to\mathbb{R}^2$ be the change of variables given in Section~\ref{Sec3.2}. Similarly let $\Psi_1\colon W_1\subset B_1\to\mathbb{R}^2$ and $\Psi_{\overline{1}}\colon W_{\overline{1}}\subset B_1\to\mathbb{R}^2$ be the change of variables associated to $p_1$ and $\overline{p_1}$ respectively. See Figure~\ref{Fig6}. 
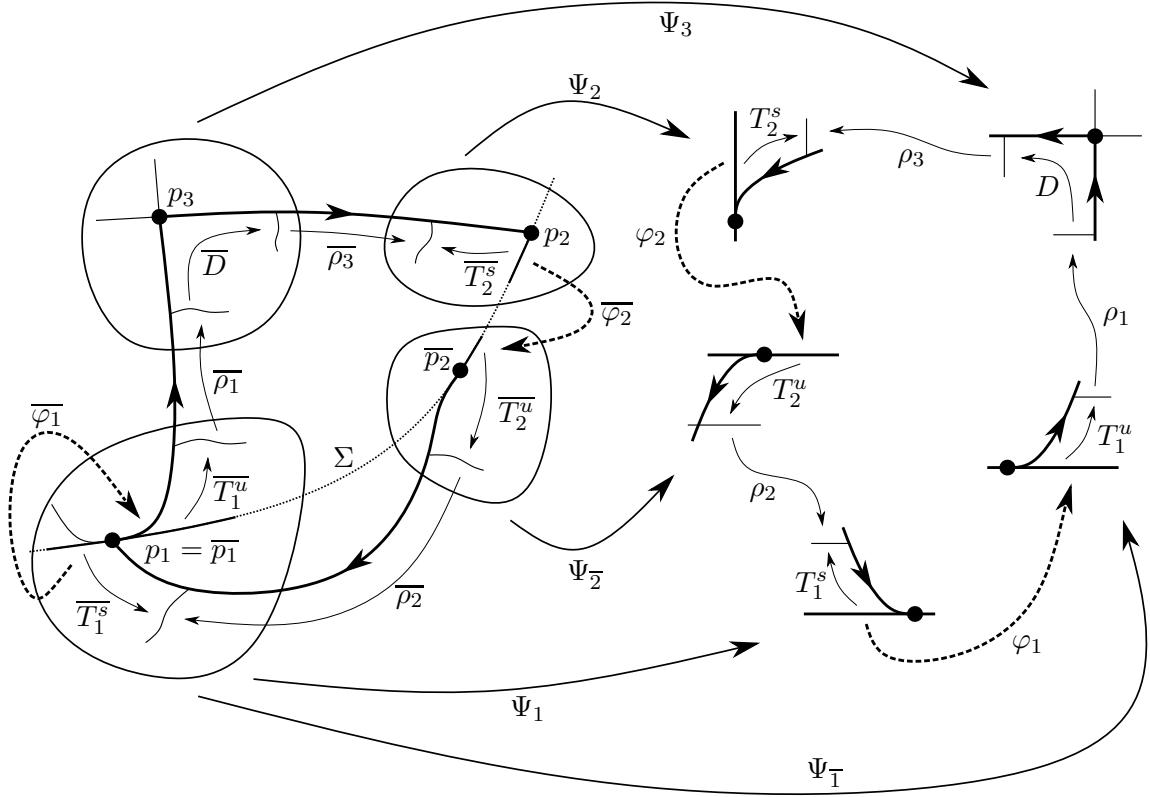
\begin{figure}[ht]
	\begin{center}
		\begin{overpic}[width=15cm]{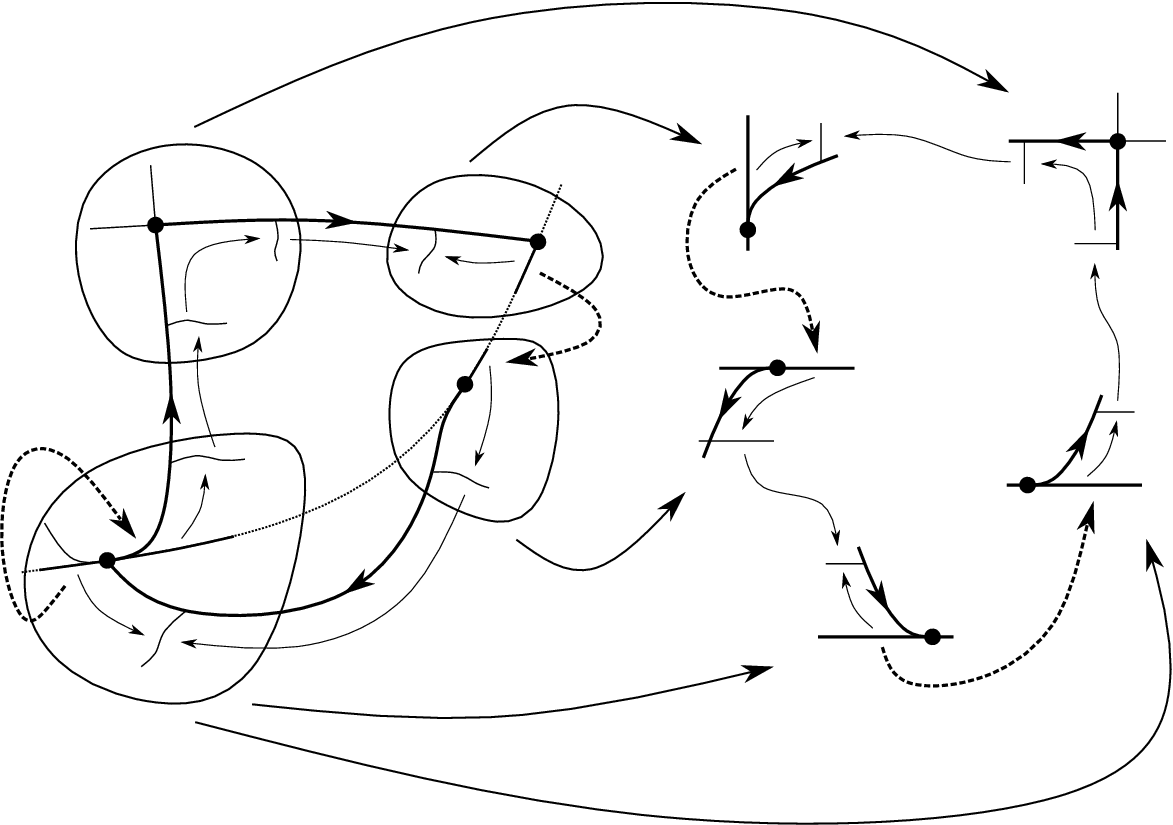}
			%\begin{overpic}[width=15cm,grid,tics=5]{Fig5.eps} %Esse é só para colocar os \put
			\put(14,52.5){$p_3$}
			\put(47,49){$p_2$}
			\put(36.5,38){$\overline{p_2}$}
			\put(12,21){$p_1=\overline{p_1}$}
			\put(17,46){$\overline{D}$}
			\put(40,45){$\overline{T_2^s}$}
			\put(43,33){$\overline{T_2^u}$}
			\put(6,15){$\overline{T_1^s}$}
			\put(18,26){$\overline{T_1^u}$}
			\put(28,47){$\overline{\rho_3}$}
			\put(34,17){$\overline{\rho_2}$}
			\put(18,36){$\overline{\rho_1}$}
			\put(52,42){$\overline{\varphi_2}$}
			\put(2,33){$\overline{\varphi_1}$}
			\put(57,67.5){$\Psi_3$}
			\put(49,62){$\Psi_2$}
			\put(49,19){$\Psi_{\overline{2}}$}
			\put(44,7){$\Psi_1$}
			\put(70,1.5){$\Psi_{\overline{1}}$}
			\put(90,53){$D$}
			\put(65,59){$T_2^s$}
			\put(67,35){$T_2^u$}
			\put(69,18){$T_1^s$}
			\put(95.5,31){$T_1^u$}
			\put(78,56){$\rho_3$}
			\put(65,26.5){$\rho_2$}
			\put(96,42){$\rho_1$}
			\put(55,49){$\varphi_2$}
			\put(88,13){$\varphi_1$}
			\put(28.5,29){$\Sigma$}
		\end{overpic}
	\end{center}
	\caption{Illustration of the maps used in the proof of Theorem~\ref{T1}.}\label{Fig6}
\end{figure}

Let $T_i^{s,u}\colon\sigma_i^{s,u}\to\tau_i^{s,u}$, $i\in\{1,2\}$ and $D\colon\sigma_3\to\tau_3$ be the transition maps and the Dulac map given by Section~\ref{Sec3.2} and \ref{Sec3.1} respectively. Let
	\[\overline{\sigma}_i^{s}=\Psi_i^{-1}(\sigma_i^{s}), \quad  \overline{\sigma}_i^{u}=\Psi_{\overline{i}}^{-1}(\sigma_i^{u}), \quad  \overline{\sigma}_3=\Psi_3^{-1}(\sigma_3), \quad \overline{\tau}_3=\Psi_3^{-1}(\tau_3), \quad i\in\{1,2\}\]
and define $\overline{T_i^{s,u}}\colon\overline{\sigma}_i^{s,u}\to\overline{\tau}_i^{s,u}$, $i\in\{1,2\}$ and $\overline{D}\colon\overline{\sigma}_3\to\overline{\tau}_3$ by
	\[\overline{T_i^s}=\Psi_i^{-1}\circ T_i^s\circ\Psi_i, \quad \overline{T_i^u}=\Psi_{\overline{i}}^{-1}\circ T_i^u\circ\Psi_{\overline{i}}, \quad \overline{D}=\Psi_3^{-1}\circ D\circ\Psi_3, \quad i\in\{1,2\}.\]
See Figure~\ref{Fig6}. Let $\overline{\rho_3}\colon\overline{\tau}\to\overline{\tau}_2^s$, $\overline{\rho_2}\colon\overline{\tau}_2^u\to\overline{\tau}_1^s$ and $\overline{\rho_1}\colon\overline{\tau}_1^u\to\overline{\sigma_3}$ be given by the flow of $X_1$ and $X_2$. Let also $\rho_3\colon\tau\to\tau_2^s$, $\rho_2\colon\tau_2^u\to\tau_1^s$ and $\rho_1\colon\tau_1^u\to\sigma_3$ be given by
	\[\rho_3=\Psi_2\circ\overline{\rho_3}\circ\Psi_3^{-1}, \quad \rho_2=\Psi_1\circ\overline{\rho_2}\circ\Psi_{\overline{2}}^{-1}, \quad \rho_1=\Psi_3\circ\overline{\rho_1}\circ\Psi_{\overline{1}}^{-1}.\]
See Figure~\ref{Fig6}. Let $\overline{\varphi_i}\colon\overline{\sigma}_i^s\to\overline{\sigma}_i^u$, $i\in\{1,2\}$, be given by the restriction of $\overline{\varphi}\colon\Sigma\to\Sigma$. Finally let $\varphi_i\colon\sigma_i^s\to\sigma_i^u$, $i\in\{1,2\}$ be given by
	\[\varphi_i=\Psi_{\overline{i}}\circ\overline{\varphi_i}\circ\Psi_i^{-1}.\]
See Figure~\ref{Fig6}. Observe that we can identify each transversal section above  (i.e. $\sigma_i^{s,u}$, $\tau_i^{s,u}$ and so on) with the interval $I=\{x\in\mathbb{R}\colon 0\leqslant x<\delta\}$, such that in these coordinates the intersection of the transversal section with the polycycle is given by $x=0$. Hence each one of the functions above (i.e. $T_i^{s,u}$, $\overline{T_i^{s,u}}$, $\rho_i$, $\overline{\rho_i}$, etc) can be seen as a function $f\colon I\to I$ such that $f(0)=0$ and $f(x)\to0$ as $x\to0$. Moreover, since each of these functions is given by the flow of a smooth vector field in a particular neighborhood, we can also assume that $f$ is strictly increasing. Therefore, from now on in this proof we assume this. Let $\nu<0<\lambda$ be the eigenvalues associated with the hyperbolic saddle $p_3$. Given $\varepsilon>0$ let $r_{3,\varepsilon}=\frac{|\nu|}{\lambda+\varepsilon}$. It follows from Proposition~\ref{P1} that for every $\varepsilon>0$ small there is $C(\varepsilon)>0$ such that
\begin{equation}\label{4}
	D(x)<C(\varepsilon)x^{r_{3,\varepsilon}}.
\end{equation}
Let $n_{i,s}$ and $n_{i,u}$, $i\in\{1,2\}$ be the stable and unstable contact orders of $p_1$ and $p_2$ respectively. It follows from Section~\ref{Sec3.2} that
\begin{equation}\label{5}
	T_i^s(x)=C_{i,s}x^{n_{i,s}}+O(x^{n_{i,s}+1}), \quad T_i^u(x)=C_{i,u}x^{n_{i,u}}+O(x^{n_{i,u}+1})
\end{equation}
with $C_{i,s}\neq0$ and $C_{i,u}\neq0$, $i\in\{1,2\}$. Suppose first that both $p_1$ and $p_2$ have finite power order and let $k_i\in\mathbb{N}$ be the power order of $p_i$, $i\in\{1,2\}$, that is, $k_i\in\mathbb{N}$ is such that
\begin{equation}\label{6}
	\overline{\varphi_i}(x)=\left(\sum_{k=k_i}^{\infty}a_kx^k\right)+R_i(x),
\end{equation}
where $a_{k_i}\neq0$ and $R_i$ is a flat $C^\infty$-function such that $R_i(0)=0$, $i\in\{1,2\}$. Finally since $\Psi_i$, $\Psi_{\overline{j}}$ and $\overline{\rho_i}$, $i\in\{1,2,3\}$, $j\in\{1,2\}$, are diffeomorphisms, we observe that
\begin{equation}\label{7}
	\rho_i(x)=c_ix+O(x^2), 
\end{equation}
with $c_i\neq0$, $i\in\{1,2,3\}$. We recall that the \emph{hyperbolicity ratio} of $p_1$, $p_2$, and $p_3$ and the \emph{graphic number} of $\Gamma^3$ are given by
	\[r_1=\frac{n_{1,u}}{n_{1,s}}k_1, \quad r_2=\frac{n_{2,u}}{n_{2,s}}k_2, \quad r_3=\frac{|\nu|}{\lambda}, \quad r=r_1r_2r_3.\]
Moreover, observe that $r_{3,\varepsilon}\to r_3$ as $\varepsilon\to0$. Hence if we let $\pi\colon\sigma_3\to\sigma_3$ be given by
\begin{equation}\label{20}
	\pi=\rho_1\circ T_1^u\circ\varphi_1\circ \bigl(T_1^s\bigr)^{-1}\circ\rho_2\circ T_2^u\circ\varphi_2\circ\bigl(T_2^s\bigr)^{-1}\circ\rho_3\circ D,
\end{equation}
then it follows from \eqref{4}, \eqref{5}, \eqref{6} and \eqref{7} that for $\varepsilon>0$ small enough we have
\begin{equation}\label{8}
\begin{array}{rl}
	\displaystyle \pi(x) =&\displaystyle \rho_1\circ T_1^u\circ\varphi_1\circ \bigl(T_1^s\bigr)^{-1}\circ\rho_2\circ T_2^u\circ\varphi_2\circ\bigl(T_2^s\bigr)^{-1}\circ\rho_3\circ D(x) \vspace{0.2cm} \\
	<&\displaystyle \rho_1\circ T_1^u\circ\varphi_1\circ \bigl(T_1^s\bigr)^{-1}\circ\rho_2\circ T_2^u\circ\varphi_2\circ\bigl(T_2^s\bigr)^{-1}\circ\rho_3\bigl(C(\varepsilon)x^{r_{3,\varepsilon}}\bigr) \vspace{0.2cm} \\
	=&\displaystyle \rho_1\circ T_1^u\circ\varphi_1\circ \bigl(T_1^s\bigr)^{-1}\circ\rho_2\circ T_2^u\circ\varphi_2\circ\bigl(T_2^s\bigr)^{-1}\bigl(C_1(\varepsilon)x^{r_{3,\varepsilon}}+\dots\bigr) \vspace{0.2cm} \\
	=&\displaystyle \rho_1\circ T_1^u\circ\varphi_1\circ \bigl(T_1^s\bigr)^{-1}\circ\rho_2\circ T_2^u\circ\varphi_2\bigl(C_2(\varepsilon)x^{\frac{1}{n_{2,s}}r_{3,\varepsilon}}+\dots\bigr) \vspace{0.2cm} \\
	=&\displaystyle \rho_1\circ T_1^u\circ\varphi_1\circ \bigl(T_1^s\bigr)^{-1}\circ\rho_2\circ T_2^u\bigl(C_3(\varepsilon)x^{\frac{1}{n_{2,s}}k_2r_{3,\varepsilon}}+\dots\bigr) \vspace{0.2cm} \\
	=&\displaystyle \rho_1\circ T_1^u\circ\varphi_1\circ \bigl(T_1^s\bigr)^{-1}\circ\rho_2\bigl(C_4(\varepsilon)x^{r_2r_{3,\varepsilon}}+\dots\bigr) \vspace{0.2cm} \\
	\vdots & \\
	=&\displaystyle K(\varepsilon)x^{r_1r_2r_{3,\varepsilon}}+ O(x^{r_1r_2r_{3,\varepsilon}+1}).
\end{array}
\end{equation}
Hence if we let $r(\varepsilon)=r_1r_2r_{3,\varepsilon}$, then we conclude that
	\[0\leqslant \pi(x)< K(\varepsilon)x^{r(\varepsilon)}+O(x^{r(\varepsilon)+1}),\]
with $r(\varepsilon)\to r$ as $\varepsilon\to0$. Therefore if $r>1$, then there is $\varepsilon_0>0$ small enough such that $r(\varepsilon_0)>1$. Hence
\begin{equation}\label{17}
	\frac{\pi(x)}{x}< K(\varepsilon_0)x^{r(\varepsilon_0)-1}+O(x^{r(\varepsilon_0)}).
\end{equation}
Since $r(\varepsilon_0)-1>0$, we have $x^{r(\varepsilon_0)-1}\to0$ as $x\to0^+$ and thus from \eqref{17} it follows $\pi(x)/x<1$ for $x>0$ small enough. Which in turn implies that
\begin{equation}\label{19}
	\pi(x)<x,
\end{equation}
for $x>0$ small enough. We now use property \eqref{19} of the map \eqref{20}, which takes place in the normal forms of the singularities of $\Gamma^3$, to study the `real' first return map $\overline{\pi}\colon\overline{\sigma}_3\to\overline{\sigma}_3$ given by
\begin{equation}\label{21}
	\overline{\pi}=\overline{\rho_1}\circ\overline{T_1^u}\circ\overline{\varphi_1}\circ \bigl(\overline{T_1^s}\bigr)^{-1}\circ\overline{\rho_2}\circ \overline{T_2^u}\circ\overline{\varphi_2}\circ\bigl(\overline{T_2^s}\bigr)^{-1}\circ\overline{\rho_3}\circ\overline{D},
\end{equation}
which actually takes place on the polycycle $\Gamma^3$, see Figure~\ref{Fig6}. It follows from \eqref{20}, \eqref{21} and from the definitions of the maps $T_i^{s,u}$, $\overline{T_i^{s,u}}$, $\rho_i$, $\overline{\rho_i}$, $D$, $\overline{D}$, etc that $\overline{\pi}=\Psi_3^{-1}\circ\pi\circ\Psi_3$. Therefore, if $\overline{x}\in\overline{\sigma}_3$ with $\overline{x}>0$ is small enough (i.e. if $\overline{x}$ is sufficiently near the polycycle), then $\Phi_3(\overline{x})=x\in\sigma_3$, $x>0$ is small enough and thus from \eqref{19} we have that
	\[\overline{\pi}(\overline{x})=\Psi_3^{-1}\circ\pi\circ\Psi_3(\overline{x})<\Psi_3^{-1}\circ\Psi_3(\overline{x})=\overline{x}.\]
So, we arrive in $\overline{\pi}(\overline{x})<\overline{x}$ for $\overline{x}>0$ small enough. Hence the polycycle is stable. The case $r<1$ follows similarly. We need only to replace \eqref{4} by
	\[C(\varepsilon)x^{s_{3,\varepsilon}}<D(x),\]
where $s_{3,\varepsilon}=\frac{|\nu|}{\lambda-\varepsilon}$ (see Proposition~\ref{P1}) and thus we obtain
	\[\pi(x)>K(\varepsilon_0)x^{s(\varepsilon_0)}+O(x^{s(\varepsilon_0)+1}),\]
for some fixed $\varepsilon_0>0$ small enough. Hence
\begin{equation}\label{9}
	\frac{\pi(x)}{x}>K(\varepsilon_0)x^{s(\varepsilon_0)-1}+O(x^{s(\varepsilon_0)}).
\end{equation}
Since $s(\varepsilon)>0$ and $s(\varepsilon)-1<0$, it follows from \eqref{9} that $\pi(x)>x$ for $x>0$ small enough and thus the polycycle is unstable. We now work with the case that some jump singularity has an infinity power order. Without loss of generality suppose that $p_1$ has infinity order. It follows from \eqref{6} that $\overline{\varphi_1}(x)=R_1(x)$, where $R_1$ is a flat function such that $R_1(0)=0$. Therefore, it follows that for every $k\in\mathbb{N}$ we have
\begin{equation}\label{10}
	\overline{\varphi_1}(x)<x^k,
\end{equation}
provided $x>0$ is small enough. Replacing \eqref{10} at \eqref{8} for a sufficiently large $k\in\mathbb{N}$, we obtain that $\pi(x)<x$ and thus the polycycle is stable. This finishes the proof. \end{proof}

\begin{remark}
	It follows from the proof of Theorem~\ref{T1} that the result also holds for other frameworks of hybrid systems, such as the \emph{Impulsive Discontinuous Dynamical Systems} (see Samoilenko and Perestyuk~\cite[Section $1.4$]{SamPer1995}) and the \emph{Impact Hybrid Systems} (see di Bernardo~\cite[Section $2.2.4$]{Bernardo2008}).
\end{remark}

\section{Example and applications}

Based on the Example~\ref{Ex1} of the bouncing ball model, we now present an example of model in which we can apply our main result.

\begin{example}
Consider the motion of a ball on a table with concave surface. Suppose that the surface of the table is given by the graph of a smooth function $f\colon[-1,1]\to\mathbb{R}$ such that $x=0$ is the only zero of $f'(x)$. Suppose also that at $x=-1$ and $x=1$ (i.e. the boundaries of the table) we have two walls preventing the ball to roll off the table. See Figure~\ref{Fig13}.	
\begin{figure}[ht]
	\begin{center}
		\begin{minipage}{8cm}
			\begin{center} 
				\begin{overpic}[height=5cm]{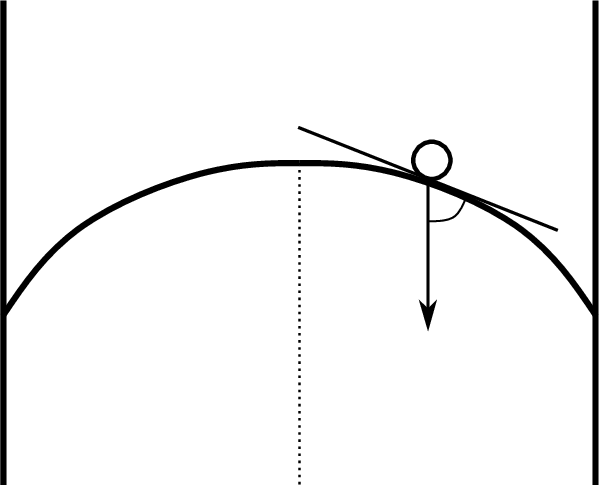} 
				%\begin{overpic}[height=5cm,grid,tics=5]{Fig23.eps} 
					\put(-8,-5){$x=-1$}
					\put(92,-5){$x=1$}
					\put(42,-5){$x=0$}
					\put(73,38){$\theta(x)$}
					\put(66,25){$g$}
				\end{overpic}
				
				$\;$
				
				$(a)$
			\end{center}
		\end{minipage}
		\begin{minipage}{8cm}
			\begin{center} 
				\begin{overpic}[height=5cm]{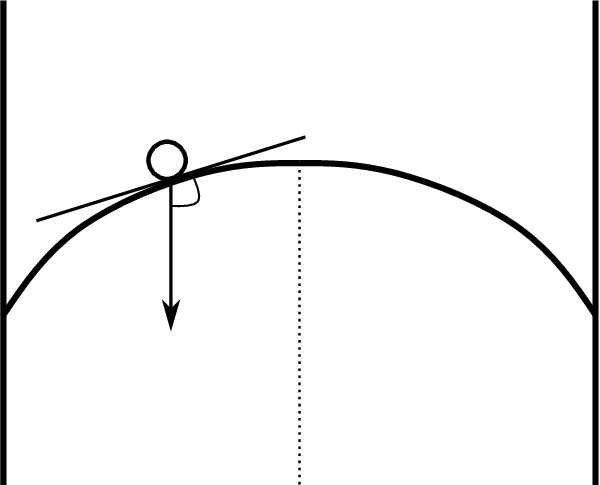} 
				%\begin{overpic}[height=5cm,grid,tics=5]{Fig24.eps} 
					\put(-8,-5){$x=-1$}
					\put(92,-5){$x=1$}
					\put(42,-5){$x=0$}
					\put(34,44){$\theta(x)$}
					\put(31,26){$g$}
				\end{overpic}
				
				$\;$
				
				$(b)$
			\end{center}
		\end{minipage}
	\end{center}
	\caption{Illustration of the model, with $(a)$ $x>0$ and $(b)$ $x<0$. Observe that $\theta(x)<\pi/2$ if $x>0$ and $\theta(x)>\pi/2$ if $x<0$.}\label{Fig13}
\end{figure}
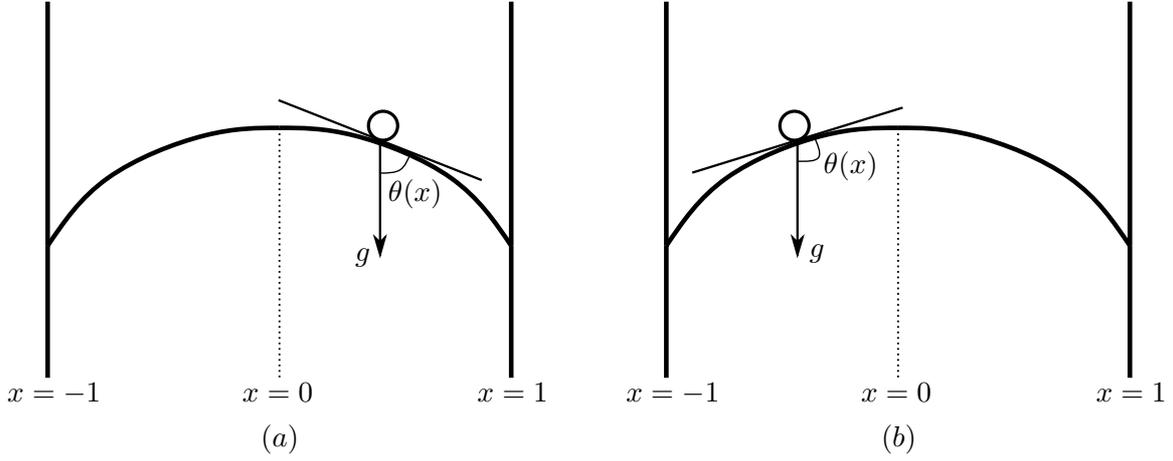
Let $x\in[-1,1]$ denote the position of the ball and $v\in\mathbb{R}$ be its initial velocity, where $v>0$ (resp. $v<0$) means that the ball is moving to the right-hand side (resp. left-hand side) of the table. If the ball is under the acceleration of constant gravity $g>0$, then the state of the ball if governed by the system of differential equations,
\begin{equation}\label{12}
	\dot x=v, \quad \dot v=g\cos\theta(x),
\end{equation}
where $\theta\colon[-1,1]\to\mathbb{R}$ is the angle between the vertical axis and the tangent line of the table at the ball, in the counterclockwise direction. See Figure~\ref{Fig13}. For $x\in(-1,1)$ and $v\in\mathbb{R}$, The Jacobian matrix of \eqref{12} is given by
\begin{equation}\label{13}
	J(x,v)=\left(\begin{array}{cc} 0 & 1 \\ \vspace{0.2cm} -g\sin\theta(x)\theta'(x) & 0 \end{array}\right).
\end{equation}
Knowing that $f'(x)$ is the slope between the tangent line and the horizontal axis, it is not hard to see that
\begin{equation}\label{14}
	\theta(x)=\frac{\pi}{2}+\arctan f'(x).
\end{equation}
Hence, it follows from \eqref{12}, \eqref{13}, \eqref{14} and $f'(0)=0$ that the origin is the unique singularity and its Jacobian matrix is given by,
\begin{equation}\label{16}
	J(0,0)=\left(\begin{array}{cc} 0 & 1 \\ \vspace{0.2cm} -gf''(0) & 0 \end{array}\right).
\end{equation}
Since $f$ is concave, it follows that $f''(0)<0$ and thus we have $\det J(0,0)=gf''(0)<0$. That is, the origin is a hyperbolic saddle. See Figure~\ref{Fig14}$(a)$.
\begin{figure}[ht]
	\begin{center}
		\begin{minipage}{8cm}
			\begin{center} 
				\begin{overpic}[height=5cm]{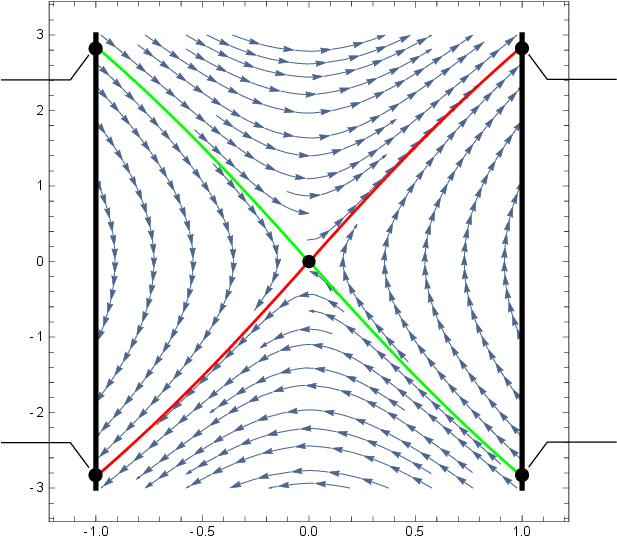} 
				%\begin{overpic}[height=5cm,grid,tics=5]{Fig25x.eps} 
					\put(95,77){$p_1$}
					\put(95,18){$p_2$}
					\put(-2,18){$p_3$}
					\put(-2,77){$p_4$}	
				\end{overpic}

				$(a)$
			\end{center}
		\end{minipage}
		\begin{minipage}{8cm}
			\begin{center} 
				\begin{overpic}[height=5cm]{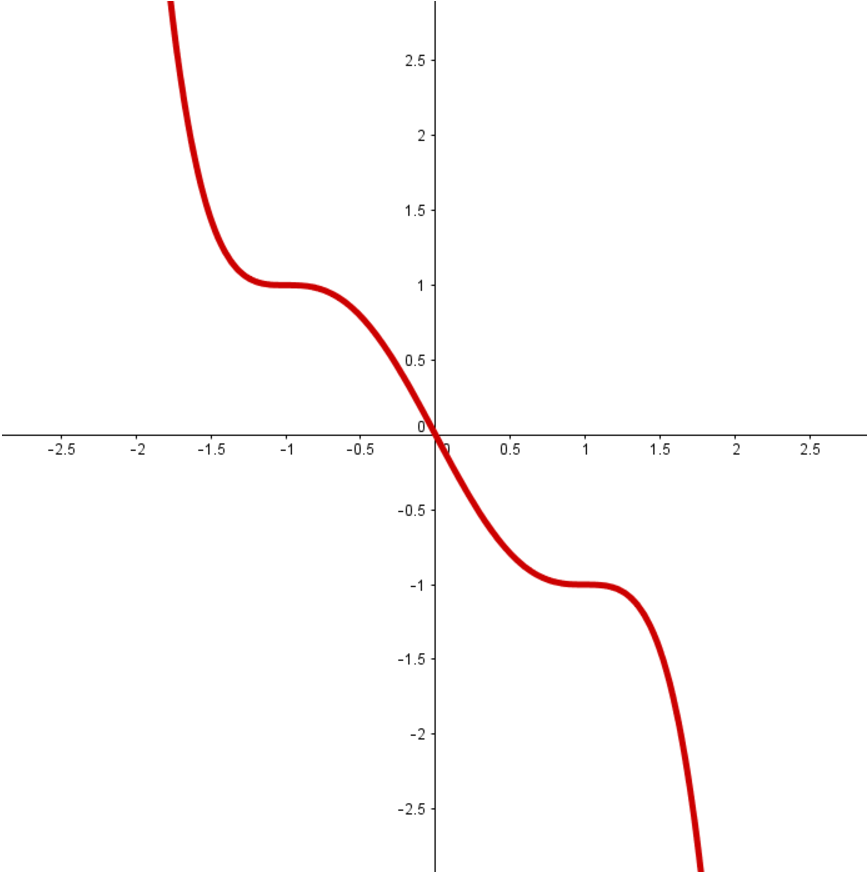} 
				%\begin{overpic}[height=5cm,grid,tics=5]{Fig26.eps} 

				\end{overpic}

				$(b)$
			\end{center}
		\end{minipage}
	\end{center}
	\caption{Illustration of $(a)$ the phase portrait of system \eqref{12} with $g=9.8$ and $f(x)=-\frac{1}{2}x^2$ and $(b)$ the map $\varphi$.}\label{Fig14}
\end{figure}
Observe that system \eqref{12} is reversible in relation to the lines $x=0$ and $v=0$. Hence, the stable and unstable manifolds of the origin intersects the lines $x=-1$ and $x=1$ symmetrically in the points $p_1=(1,v_0)$, $p_2=(1,-v_0)$, $p_3=(-1,-v_0)$ and $p_4=(-1,v_0)$, for some $v_0>0$. See Figure~\ref{Fig14}$(a)$. Rescaling the variable $v$ if necessary, we can assume $v_0=1$.

We can now suppose that both walls act like a pinball machine and thus when the ball hits the walls,it is tossed back to the table. More precisely we let the map $\varphi\colon\mathbb{R}\to\mathbb{R}$ be given by
\begin{equation}\label{15}
	\varphi(v)=-\frac{15}{8}v+\frac{5}{4}v^3-\frac{3}{8}v^5,
\end{equation}
represents the gain of energy of the ball when hitting each of the walls, see Figure~\ref{Fig14}$(b)$. Observe that $\varphi(1)=-1$ and $\varphi(-1)=1$ and thus $\varphi$ sends $p_1$ to $p_2$ and $p_3$ to $p_4$. In particular our hybrid system has a compound polycycle given by a figure in an eight shape. See Figure~\ref{Fig15}.
\begin{figure}[h]
	\begin{center}
		\begin{overpic}[width=12cm]{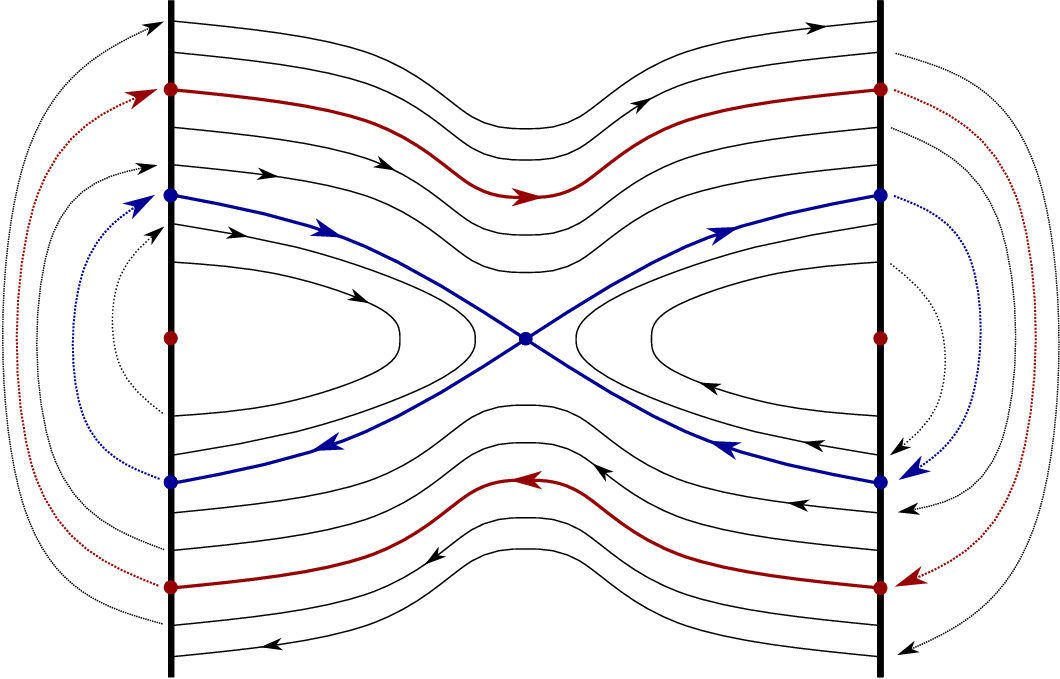}
		%\begin{overpic}[width=12cm,grid,tics=5]{Fig27xxx.eps} %Esse é só para colocar os \put
			\put(48.5,34){$\mathcal{O}$}
			\put(84,47){$p_1$}
			\put(79.5,16.75){$p_2$}
			\put(12.5,16.5){$p_3$}
			\put(17.5,46.5){$p_4$}
			\put(70,55.25){$\gamma$}
			\put(25,7.25){$\gamma$}
		\end{overpic}
	\end{center}
	\caption{Illustration of the hybrid system given by \eqref{12} and \eqref{15}. Blue means stable and red means unstable. Color available on the online version.}\label{Fig15}
\end{figure}
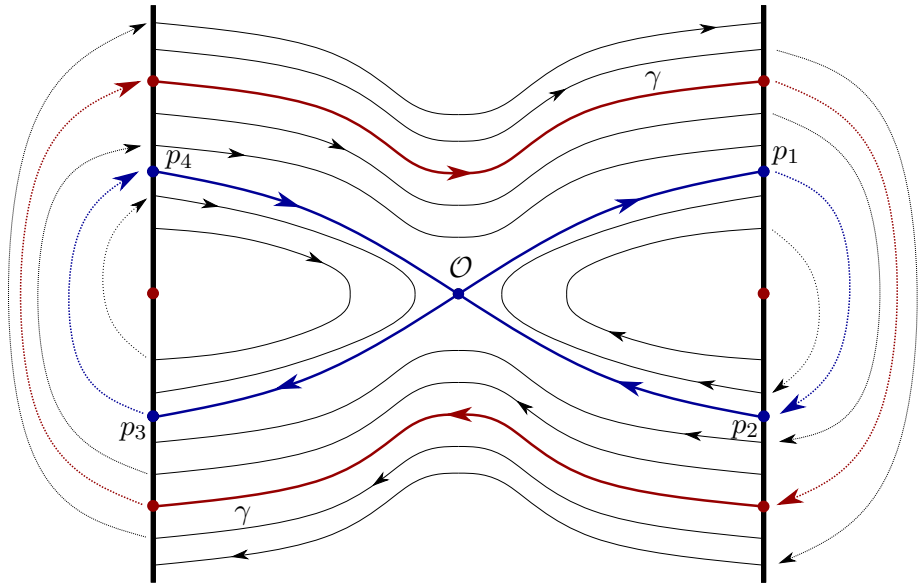
It follows from \eqref{16} that the hyperbolicity ratio of the origin is given by $r_0=1$. From \eqref{15} it is not hard to see that $\varphi'(\pm1)=\varphi''(\pm1)=0$ and $\varphi'''(\pm1)\neq0$. Hence, the hyperbolicity ratio of the jump singularities $(p_1,p_2)$ and $(p_3,p_4)$ are given by $r_L=3$ and $r_R=3$, respectively. 

Therefore, it follows from Theorem~\ref{T1} that the polycycles 
	\[\Gamma_L=\{\mathcal{O},(p_3,p_4)\}, \quad \Gamma_R=\{\mathcal{O},(p_1,p_2)\}, \quad \Gamma=\{\mathcal{O},(p_1,p_2),\mathcal{O},(p_3,p_4)\},\]
are stable. From \eqref{15} it is not hard to see that the orbits hitting the wall with velocity $|v|>0$ big enough spirals to the infinity, i.e., after a certain threshold the ``pinball'' walls feed one another and the balls keep going between left and right while the velocity increases. Hence if we consider a first return map $\pi\colon\ell\to\ell$, with $\ell=\{0\}\times\mathbb{R}_{v>0}$, then it follows from the Intermediate Value Theorem that we have at least one periodic orbit $\gamma$. See Figure~\ref{Fig15}. 
\end{example}

\section*{Acknowledgments}

We thank to the reviewers their comments and suggestions which help us to improve the presentation of this paper. The authors are partially supported by S\~ao Paulo Research Foundation (FAPESP), grants 2019/10269-3, 2021/01799-9 and 2022/14353-1, and by Coordination of Superior Level Staff Improvement (CAPES-Brazil), grant 88887.475751/2020-00.

\end{document}